\documentclass[12pt, a4paper]{amsart}

\usepackage[top=37mm,bottom=37mm,left=27mm,right=27mm]{geometry}
\usepackage{amsthm,amssymb,enumerate,xcolor}
\usepackage{graphics}
\usepackage[hidelinks,colorlinks=true,linkcolor=blue,citecolor=teal,linkcolor=green!50!black]{hyperref}
\usepackage{array}
\usepackage{booktabs}

\newcommand{\g}{\mathfrak}
\newcommand{\s}{\mathsf}
\newcommand{\R}{\mathbb{R}}
\newcommand{\C}{\mathbb{C}}
\renewcommand{\H}{\mathbb{H}}
\renewcommand{\O}{\mathbb{O}}
\newcommand{\F}{\mathbb{F}}
\newcommand{\ad}{\operatorname{ad}}
\newcommand{\Ad}{\operatorname{Ad}}
\newcommand{\tr}{\operatorname{tr}}

\newcommand{\Aut}{\operatorname{Aut}}
\newcommand{\End}{\operatorname{End}}

\newcommand{\Exp}{\operatorname{Exp}}
\newcommand{\id}{\operatorname{id}}

\newcommand{\ext}[1]{\ensuremath{\scalebox{0.9}{$\bigwedge$}^{#1}\,}}

\newtheorem{theorem}{Theorem}[section]
\newtheorem{lemma}[theorem]{Lemma}
\newtheorem{proposition}[theorem]{Proposition}

\theoremstyle{definition}

\newtheorem{maintheorem}{Theorem}

\newtheorem{maincorollary}[maintheorem]{Corollary}

\begin{document}
\title{Polar homogeneous foliations on symmetric spaces of rank one}
\author[J.~C.~D\'iaz-Ramos]{Jos\'e Carlos D\'iaz-Ramos}
\address{CITMAGA, Universidade de Santiago de Compostela, Spain}
\email{josecarlos.diaz@usc.es}

\author[J.~M.~Lorenzo-Naveiro]{Juan Manuel Lorenzo-Naveiro}
\address{Department of Mathematics, University of Oklahoma, USA}
\email{jnaveiro@ou.edu}

\begin{abstract}
We classify polar homogeneous foliations on rank one symmetric spaces of noncompact type up to orbit equivalence.
\end{abstract}



\thanks{The authors have been supported by 
	grant PID2022-138988NB-I00 funded by MICIU/AEI/10.13039/501100011033 and by ERDF, EU, 
	and by project ED431C 2023/31 (Xunta de Galicia, Spain).}

\maketitle

\section{Introduction}

An isometric action of a Lie group $G$ on a Riemannian manifold $M$ is polar if it admits a section,
that is, a connected submanifold $\Sigma$ that meets every orbit of the action orthogonally.
We also say that the $G$-action is hyperpolar if it admits a flat section.
The archetypal example of a polar action is the standard representation of the orthogonal group
$\s{O}(2)$ on the Euclidean plane $\mathbb{R}^2$.
Here, the sections are precisely the straight lines that go through the origin.
Similarly, the action of $\s{O}(n)$ on the vector space $\mathrm{Sym}_{n}(\mathbb{R})$
of all $n\times n$ symmetric matrices is also polar and
the subspace of all diagonal matrices is a section.
In view of this example, the elements of a section are commonly understood as canonical forms
of the elements of the ambient space.
Polar actions (and in particular cohomogeneity one actions) have been extensively used
in invariant theory~\cite{PalaisTerng87} as well as in the construction of new examples
of manifolds with exceptional holonomy~\cite{Bryant97} or inhomogeneous Einstein
metrics~\cite{Bohm98}, to name a few examples.
For a general treatment on the main properties of polar actions, we refer the reader
to~\cite{GroveZiller,PalaisTerng87}.

In this article, we are interested in classifying polar actions on Riemannian manifolds
up to orbit equivalence.
The natural choice of ambient spaces in which to tackle this problem are symmetric spaces.
This is due to two reasons:
on the one hand, symmetric spaces have large isometry groups with a well understood structure;
on the other hand, a section of a polar action is a totally geodesic submanifold, and those can
be characterized algebraically in symmetric spaces by means of so-called Lie triple systems.

The first systematic classification of polar actions was carried out by Dadok~\cite{Dadok85},
who showed that every polar representation is orbit equivalent to the isotropy representation
of a symmetric space.
This result also yields the classification of polar actions on round spheres.
Later, Wu~\cite{Wu} derived the classification of polar actions on real hyperbolic spaces.

There have been significant advances in the classification of polar actions on (irreducible)
symmetric spaces of compact type.
Indeed, Podestà and Thorbergsson~\cite{PodestaThorbergsson99} described all polar actions
on symmetric spaces of compact type and rank one
(that is, spheres and projective spaces over the normed real division algebras).
From Dadok's result, one can relate these actions to Cartan's classification of symmetric
spaces.
This work was later revised by Gorodski and Kollross~\cite{GorodskiKollross16},
who found a missing example in the case of the Cayley projective plane.
Moreover, Biliotti~\cite{Biliotti06} conjectured that if $M$ is an irreducible symmetric
space of compact type and rank greater than one, then every polar action on $M$ is hyperpolar.
This conjecture was shown to be true by Kollross and Lytchak~\cite{KollrossLytchak13}.
Additionally, Kollross proved that every hyperpolar action on $M$ is
either of cohomogeneity one or orbit equivalent to a Hermann action.
The proof of this fact, as well as the classification of both families of actions,
can be found in~\cite{Kollross02}.

In contrast with the compact case, results concerning polar actions on symmetric spaces
of noncompact type are quite scarce.
The only such spaces in which polar actions have been completely determined
are the real hyperbolic space (as mentioned above) and the complex
hyperbolic space~\cite{DiazDominguezKollross20}.
The special case of cohomogeneity one actions has enjoyed the most attention.
Their classification is the result of a collective effort that started
in~\cite{BerndtBruck01} and finally ended in~\cite{SanmartinSolonenko25}.

We say that an isometric action induces a homogeneous foliation if all of its
orbits have the same type
(that is, if all its isotropy subgroups are conjugate).
Berndt and Tamaru classified all codimension one homogeneous foliations
on irreducible symmetric spaces of noncompact type—see~\cite{BerndtTamaru03}
and the subsequent generalization~\cite{Solonenko21} to the reducible case.
Moreover, Berndt, Tamaru and the first author developed in~\cite{BerndtDiazTamaru10}
a method to construct all hyperpolar homogeneous foliations on irreducible symmetric
spaces of noncompact type.
More recently, we have obtained the classification of
codimension two polar homogeneous foliations on irreducible symmetric spaces of
noncompact type, see~\cite{DiazLorenzo}.
These foliations turn out to be lower dimensional analogues of the codimension one
examples constructed by Berndt and Tamaru.
The main difficulty in dealing with polar homogeneous foliations of larger codimension
is that their sections are not known to satisfy any restriction apart from being
totally geodesic.
In fact, one can see that if $M$ is an irreducible symmetric space of noncompact type,
then every boundary component of $M$ is the section of a polar homogeneous foliation
on $M$~\cite{BerndtDiazTamaru10}.
However, using the theory of maximal solvable subalgebras of real semismple Lie algebras,
one can immediately give constraints on the subgroups of the isometry group of $M$
that act polarly inducing a foliation~\cite[Proposition~4.2]{DiazLorenzo}.

The purpose of this article is to classify all polar homogeneous foliations on
symmetric spaces of noncompact type and rank one.
As we mentioned earlier, the only symmetric spaces of noncompact type on which
polar actions have been completely classified are the real and complex
hyperbolic spaces.
We also refer the reader to~\cite{berndt-diaz-ramos} for a direct proof of the
classification of polar homogeneous foliations on complex hyperbolic spaces.
Because of this, the ambient spaces under our consideration are the
quaternionic hyperbolic spaces $\H\s{H}^n$ (where $n\geq 2$)
and the Cayley hyperbolic plane $\O\s{H}^2$.
In these spaces, the works of Kollross~\cite{kollross-duality,kollross-octonions}
provide partial classifications of polar actions that preserve a proper totally
geodesic submanifold, and it turns out that all of the actions found in his results possess
singular orbits.
In addition, Berndt and Tamaru~\cite{BerndtTamaru07} derived the classification of
cohomogeneity one actions on the Cayley hyperbolic plane.
In the same paper, the authors reduce the classification of cohomogeneity one
actions on quaternionic hyperbolic spaces to that of the so-called protohomogeneous
vector subspaces of quaternionic Euclidean vector spaces.
The latter problem was solved by the first author, Dom\'{i}nguez-V\'{a}zquez and
Rodr\'{i}guez-V\'{a}zquez in~\cite{DiazDominguezRodriguez21}.
At any rate, there are no results concerning polar actions (of cohomogeneity greater than one)
without singular orbits on $\H\s{H}^n$ or $\O\s{H}^{2}$
(or more generally, actions that do not preserve a proper totally geodesic submanifold).

We now state the main theorems that will be proved in this article.
Let $M$ denote a rank one Riemannian symmetric space of noncompact type
that can be represented by a symmetric pair $(G,K)$.
More explicitly, $M=\F\s{H}^n$ is a hyperbolic space over one of the real normed
division algebras $\mathbb{F}\in\{\R,\C,\H,\O\}$, with $n=2$ if $\F=\O$.
Throughout this paper, we work with the symmetric metric on $\F\s{H}^n$ that pinches its
sectional curvature between $-1$ and $-1/4$.
Let $\g{g}=\g{k}\oplus\g{a}\oplus\g{n}$ be an Iwasawa decomposition of the Lie algebra of $G$.
We can write $\g{n}=\g{g}_\alpha\oplus\g{g}_{2\alpha}$ as a direct sum of vector spaces, 
where $\g{g}_{2\alpha}$ is the center of $\g{n}$ and $\g{g}_\alpha\cong\mathbb{F}^{n-1}$ as an $\mathbb{F}$-module.
Given a subspace $\g{b}\subseteq \g{a}$ and an abelian subspace $\g{v}\subseteq \g{g}_\alpha$, we construct the subalgebra $\g{s}_{\g{b},\g{v}}=(\g{a}\ominus\g{b})\oplus(\g{n}\ominus \g{v})$.
We will show that the connected subgroup $\s{S}_{\g{b},\g{v}}$ of $G$ whose Lie algebra is 
$\g{s}_{\g{b},\g{v}}$ acts polarly on $\F\s{H}^n$ inducing a foliation.
Moreover, any section of the action of $\s{S}_{\g{b},\g{v}}$ is a real hyperbolic space whose curvature is equal to the maximum sectional curvature attained by a $2$-plane in $\F\s{H}^n$.
We will also show that the actions of $\s{S}_{\g{b}_1,\g{v}_1}$ and $\s{S}_{\g{b}_2,\g{v}_2}$ on $\mathbb{F}\s{H}^n$
are orbit equivalent if and only if $\dim\g{b}_1=\dim\g{b}_2$ and $\dim\g{v}_1=\dim\g{v}_2$.
Thus, it makes sense to define $\s{S}_{a,b}=\s{S}_{\g{b},\g{v}}$, 
where $a=\dim\g{b}\in\{0,1\}$ and $b=\dim\g{v}\in\{0,\dots,n-1\}$.
This definition is correct up to isometric congruence.
The main result of this paper is:

\begin{maintheorem}\label{th:main-short}
Every nontrivial polar homogeneous foliation on a rank one symmetric space of noncompact type is isometrically congruent to the orbit foliation induced by a group $\s{S}_{a,b}$ with $a\in\{0,1\}$ and $b\in\{0,\dots,n-1\}$.
\end{maintheorem}

Note that if $a=0$ and $b=0$, then $\s{S}_{0,0}$ acts transitively on $\F\s{H}^n$.
If $a+b=1$ then the action of $\s{S}_{a,b}$ is of cohomogeneity one~\cite{BerndtTamaru03}.
If $a=1$ and $b=0$, then $\s{S}_{1,0}$ is nilpotent and its action induces the horosphere foliation.
If $a=0$ and $b=1$, then $\s{S}_{0,1}$ is solvable and its action induces the so-called solvable foliation.
If $a=1$ we get $\g{b}=\g{a}$ and $\g{s}_{1,b}=\g{n}\ominus\g{v}$ above;
in this case all the orbits of the action of $\s{S}_{1,b}$ are congruent and are contained in horospheres of $\F\s{H}^n$.

Taking into account that the trivial action is also polar we obtain

\begin{maintheorem}
There are exactly $2n+1$ congruency classes of homogeneous polar foliations on $\F\s{H}^n$ up to isometric congruence.
\end{maintheorem}

Remarkably, the number of congruency classes of homogeneous foliations on hyperbolic
spaces does not depend on base field $\mathbb{F}$, but only on $\dim_\mathbb{F}M$.
In particular, in the Cayley hyperbolic plane $\O\s{H}^2$ we obtain two homogeneous polar foliations of codimension one,
one of codimension two, the transitive action, and the trivial action, up to isometric congruence.

In this paper we tackle quaternionic hyperbolic spaces and the Cayley hyperbolic plane.
Parts of the proof are common to both cases, but some arguments are specific of the normed division algebra under consideration.

Theorem~\ref{th:main-short} follows directly from the more precise statement given below,
which is already known to be true for real and complex hyperbolic spaces:

\begin{maintheorem}\label{th:main}
Let $\F\s{H}^n=G/K$, $n\geq 2$, be a rank one symmetric space of noncompact type.
Then the following assertions are true:
\begin{enumerate}[{\rm (1)}]
\item\label{th:main-examples} Given a vector subspace $\g{b}\subseteq \g{a}$ and an abelian subspace $\g{v}\subseteq \g{g}_\alpha$, the subgroup $\s{S}_{\g{b},\g{v}}\subseteq {G}$ acts polarly on $\F\s{H}^n$ inducing a foliation.

\item\label{th:main-classification} Any nontrivial polar homogeneous foliation on $\F\s{H}^n$ is isometrically congruent to the orbit foliation of a subgroup $\s{S}_{\g{b},\g{v}}$ as in item~(\ref{th:main-examples}).

\item\label{th:main-congruence} Given subspaces $\g{b}_1$, $\g{b}_2\subseteq \g{a}$ 
and abelian subspaces $\g{v}_1$, $\g{v}_2\subseteq \g{g}_{\alpha}$, 
the actions of $\s{S}_{\g{b}_1,\g{v}_1}$ and $\s{S}_{\g{b}_1,\g{v}_2}$ are orbit equivalent 
if and only if $\dim\g{b}_1=\dim\g{b}_2$ and $\dim \g{v}_1 = \dim \g{v}_2$.
\end{enumerate}
\end{maintheorem}

A consequence of Theorem~\ref{th:main} is that every nontrivial polar homogeneous
foliation on a hyperbolic space has a section of constant curvature.
This will be shown in Proposition~\ref{prop:examples}.
Combining this with Lemma~\ref{lemma:singular}, one can show that every nontrivial
polar action on a hyperbolic space has a section of constant curvature.
This result is also known to be true if the ambient space is an irreducible
symmetric space of compact type,
see~\cite[Proposition~1B.1]{PodestaThorbergsson99} for the rank one case
and~\cite[Theorem~1.1]{KollrossLytchak13} for the case of higher rank.
By contrast, in~\cite[Proposition~4.2]{BerndtDiazTamaru10} the authors give examples
of polar actions on irreducible symmetric spaces of noncompact type
that possess sections of nonconstant curvature.
These observations lead us to the following characterization of irreducible
symmetric spaces of noncompact type and higher rank:

\begin{maincorollary}\label{cor:rank-one-characterization}
    Let $M$ be an irreducible symmetric space.
    Then the following conditions are equivalent:
    \begin{enumerate}[\rm (i)]
        \item
            $M$ is a symmetric space of noncompact type and rank $r\geq 2$.

        \item
            $M$ admits a polar action whose sections have nonconstant curvature.
    \end{enumerate}
\end{maincorollary}

We now describe the structure of this paper.
In Section~\ref{sec:preliminaries} we present the basic facts and notation that will be needed in the rest of the paper.
In particular, we introduce generalized Heisenberg algebras in \S\ref{sec:heisenberg}.
These are used in Subsection~\ref{subsec:rank one} to deduce some formulae that are essential in our calculations.
We also recall the classification of totally geodesic submanifolds of rank one symmetric spaces (\S\ref{sec:fhn-tg-submanifolds}) and the basic facts about polar actions in Subsection~\ref{sec:polar}.
The proofs of items~(\ref{th:main-examples}) and~(\ref{th:main-congruence}) 
of Theorem~\ref{th:main} are tackled in Section~\ref{sec:examples}. 
In Section~\ref{sec:general-results}, we provide some general results concerning polar
homogeneous foliations on symmetric spaces of noncompact type and rank one.
Finally, the proof of Theorem~\ref{th:main}~(\ref{th:main-classification}) is carried out in 
Section~\ref{sec:quaternionic} for quaternionic hyperbolic spaces 
and in Section~\ref{sec:cayley} for the Cayley hyperbolic plane.

\section{Preliminaries}\label{sec:preliminaries}

The purpose of this section is to introduce the general notions and results concerning symmetric spaces of noncompact type and rank one, as well as polar actions.
The solvable model $M=AN$ allows us to view a rank one symmetric space of noncompact type as an extension of a one-dimensional Lie group by a two-step nilpotent (or abelian) Lie group.
The subgroup $N$ is an example of a generalized Heisenberg group, while $M$ belongs to the class of Damek--Ricci spaces.
By exploiting the algebraic structure of the Lie algebra $\g{n}$ as a generalized Heisenberg algebra, we are able to perform general calculations with relative ease.

\subsection{Generalized Heisenberg algebras}\label{sec:heisenberg}\hfill

A detailed account of the classification and geometry of both generalized Heisenberg groups and Damek--Ricci spaces can be found in~\cite{BerndtTricerriVanHecke}.
We also refer the reader to~\cite{cowling-dooley-koranyi-ricci} for a treatment focused on rank one symmetric spaces.

Let $\g{v}$ and $\g{z}$ be real vector spaces and $\beta\colon \ext{2} \g{v}\to \g{z}$
a skew-symmetric bilinear map,
We construct the vector space $\g{n}=\g{v}\oplus \g{z}$ and endow it with an inner product 
$\langle \cdot , \cdot \rangle_{\g{n}}$ such that $\g{v}$ and $\g{z}$ are orthogonal, 
and a Lie bracket given by $[\g{n},\g{z}]=0$ and $[U,V]=\beta(U\wedge V)$ for all $U$, $V\in \g{v}$.

Each vector $Z\in \g{z}$ gives rise to a skew-symmetric endomorphism $J_Z \in \g{so}(\g{v})$ by letting 
$\langle J_Z U, V \rangle_{\g{n}} = \langle [U,V], Z\rangle_{\g{n}}$ for each $U$, $V\in \g{z}$.
We say that $\g{n}$ is a \textit{generalized Heisenberg algebra} if 
$J_Z^2=-\lvert Z \rvert_{\g{n}} ^2 \id_{\g{v}}$ for every $Z\in \g{v}$.
In other words, we require $J\colon \g{z}\to \End(\g{v})$ to extend to an algebra homomorphism $J\colon \operatorname{Cl}(\g{z})\to \End(\g{v})$ defined on the Clifford algebra of $\g{z}$ (with the aforementioned inner product).
In particular, generalized Heisenberg algebras are closely related to Clifford modules.
The classification of Clifford modules implies the classification of generalized Heisenberg algebras~\cite[Section~3.1.2]{BerndtTricerriVanHecke}.

A consequence of this definition is that the maps $J_Z$ with $\lvert Z \rvert_{\g{n}} = 1$ become complex structures on the real vector space $\g{v}$.
Because a complex structure is also a linear isomorphism, the Lie algebra $\g{n}$ is two-step nilpotent (except in the case $\g{z}=0$, where $\g{n}$ is abelian) and its center is $\g{z}$.
A subspace $\g{w}\subseteq \g{v}$ is said to be \textit{totally real} if it is totally real with respect to all $J_Z$, where $Z\in\g{z}$ and $\lvert Z \rvert_{\g{n}} = 1$.
Equivalently, $\g{w}$ is totally real if and only if it is an abelian subspace of $\g{v}$.

We will use the following formulae involving $U$, $V\in\g{v}$ and $X$, $Y\in \g{z}$:
\begin{align*}
J_X J_Y + J_Y J_X ={}
& -2\langle X, Y \rangle_{\g{n}} \id_{\g{v}}, 
& \langle J_X U, J_X V \rangle_{\g{n}} ={}
& \lvert X \rvert ^2 \langle U, V \rangle_{\g{n}}, \\
\langle J_X U, J_Y U \rangle_{\g{n}} ={}
& \lvert U \rvert_{\g{n}}^2 \langle X, Y \rangle_{\g{n}}, 
& [V,J_X V]={}& \lvert V \rvert_{\g{n}}^2 X.
\end{align*}

Let $V\in \g{v}$ be a nonzero vector and define $\g{J}V=\{J_Z V\colon Z\in \g{z}\}$.
It is clear from the skew-symmetry of the operators $J_Z$ that $\R V$ and $\g{J}V$ are orthogonal subspaces of $\g{v}$.
Furthermore, it is easy to check that an element $W\in \g{v}$ belongs to $\g{v}\ominus(\R V\oplus \g{J}V)$ if and only if $\langle V,W \rangle_{\g{n}} = 0$ and $[V,W]=0$.

In this article, we make use of the following fact:
Suppose $(V,\langle\cdot ,\cdot\rangle)$ is a Euclidean vector space.
Then the exterior square $\ext{2}V$ becomes a Euclidean vector space
with the inner product that is characterized by
$
\langle u\wedge v, w\wedge x \rangle
=
\langle u,w\rangle\langle v,x\rangle-\langle u,x\rangle\langle v,w\rangle
$
for all $u,\,v,\,w,\,x\in V$.

\subsection{Noncompact rank one symmetric spaces}\label{subsec:rank one}\hfill

Here we gather some general facts on symmetric spaces of noncompact type, focusing mainly on those of rank one.
See~\cite{Helgason01} for a thorough treatment of symmetric spaces.
We also refer to~\cite{Eberlein96} for more information on symmetric spaces of noncompact type, as well as to~\cite[Chapter~6]{knapp} for the structure theory of real semisimple Lie groups and algebras.

Let $M$ be an $n$-dimensional symmetric space of noncompact type.
For the purposes of this article, we assume that $M$ is connected and irreducible as a Riemannian manifold.
We can present $M$ as a Riemannian homogeneous space $G/K$, where $G$ is the identity component $I^0(M)$ of the full isometry group $I(M)$ and $K=G_o$ is the isotropy subgroup corresponding to a basepoint $o\in M$.
The Lie group $G$ is simple and noncompact, whereas $K$ is a maximal compact subgroup of $G$.
Moreover, the sectional curvature of $M$ is nonpositive and the Riemannian exponential map $\exp_o\colon T_o M\to M$ is a global diffeomorphism.

Let us consider the isometry algebra $\g{g}$, which is simple.
The Killing form of $\g{g}$ is denoted by $\mathcal{B}$.
One can show that there exists an involutive automorphism $\theta\in \Aut(\g{g})$ with the property that the symmetric bilinear form $\mathcal{B}_{\theta}(X,Y)=-\mathcal{B}(X,\theta Y)$ is positive definite.
We say in this case that $\theta$ is the Cartan involution of $\g{g}$, and it is unique up to inner automorphisms.
The eigenspace decomposition $\g{g}=\g{k}\oplus\g{p}$ associated with $\theta$ (where $\g{k}=\ker(1-\theta)$ and $\g{p}=\ker(1+\theta)$) is known as the Cartan decomposition of $\g{g}$.
Because $\g{k}$ coincides with the isotropy algebra at $o$, we may identify $\g{p}$ with the tangent space $T_o M$.
The restriction of $\mathcal{B}_\theta$ to $\g{p}$ is an $\Ad(K)$-invariant inner product on $\g{p}\cong T_o M$, so that it may be extended to a $G$-invariant metric on $M$.
Since $M$ is irreducible, this $G$-invariant metric is homothetic to the original metric on $M$.
It should be noted that if a Lie group action on $M$ is polar, then it is also polar with respect to any multiple of the metric on $M$, so it will be convenient for us to work with an adequate rescaling of said metric.
We also remark that if $X\in\g{g}$ is any vector, then the adjoint of $\ad(X)$ is $\ad(X)^{*}=-\ad(\theta X)$.

We now choose a maximal abelian subspace $\g{a}$ of $\g{p}$.
Any two choices of $\g{a}$ are conjugate by $K$, so the dimension of $\g{a}$ is a well-defined invariant of $M$, known as the rank of $M$.
Let $\g{a}^{*}$ denote the dual vector space of $\g{a}$, which inherits an inner product
from the one on $\g{a}$.
Then, any linear form $\lambda \in \g{a}^{*}$ determines a vector subspace
\begin{equation*}
    \g{g}_{\lambda}
    =
    \big\{
        X\in \g{g} \colon [H,X]=\lambda(H)X \text{ for all } H\in \g{a}
    \big\}.
\end{equation*}
Note that, since the family $\ad(\g{a})$ consists of derivations,
one sees that $[\g{g}_\lambda,\g{g}_\mu]\subseteq\g{g}_{\lambda+\mu}$
whenever $\mu\in\g{a}^*$.
If both $\lambda$ and $\g{g}_{\lambda}$ are nonzero, we say that $\lambda$ is a (restricted) root of $\g{g}$ and $\g{g}_{\lambda}$ is its corresponding root space.
The set of roots is denoted by $\Sigma$.
Because the family $\ad(\g{a})\subseteq \g{gl}(\g{g})$ is a commuting family of symmetric endomorphisms with respect to $\mathcal{B}_{\theta}$, it follows that $\g{g}$ can be decomposed as the orthogonal direct sum
$
\g{g}=\g{g}_{0}\oplus \bigoplus_{\lambda\in \Sigma}\g{g}_{\lambda}.
$
This vector space decomposition is known as the root space decomposition of $\g{g}$.
Moreover, the set $\Sigma \subseteq \g{a}^{*}$ is a (perhaps nonreduced) abstract root system in the sense of \cite[Section~2.5]{knapp}.
We also set $\g{k}_{\lambda}=(1+\theta)\g{g}_{\lambda}$ and $\g{p}_{\lambda}=(1-\theta)\g{g}_{\lambda}$ for each $\lambda\in \g{a}^{*}$.
It is readily shown that $\g{g}_{0}=\g{k}_{0}\oplus \g{a}$.
In particular, the adjoint action of $\g{k}_{0}$ normalizes every root space.
The connected subgroup of $G$ with Lie algebra $\g{k}_0$ is denoted by $K_0$, and it is a compact subgroup of $K$.

Now, we fix a notion of positivity on $\Sigma$.
We let $\Sigma^{+}$ be the corresponding set of positive roots and $\Lambda \subseteq \Sigma^{+}$ be the subset of simple roots.
One sees that $\Lambda$ is a basis of $\g{a}^{*}$, and thus its cardinality is equal to the rank of $M$.
Furthermore, the vector subspace $\g{n}=\bigoplus_{\lambda \in \Sigma^{+}}\g{g}_{\lambda}$ is a nilpotent subalgebra of $\g{g}$.
The Iwasawa decomposition theorem asserts that $\g{g}=\g{k}\oplus\g{a}\oplus\g{n}$ (vector space direct sum).
In addition, by taking $A$, $N$ and $AN$ to be the connected subgroups of $G$ whose Lie algebras
are $\g{a}$, $\g{n}$ and $\g{a}\oplus\g{n}$ respectively,
one sees that the multiplication maps $K\times A\times N\to G$
and $A\times N\to AN$ are diffeomorphisms.
In particular, the evaluation map $g\in AN \to g\cdot o \in M$ is also a diffeomorphism,
and by pulling back the Riemannian metric on $M$ to $AN$
we conclude that $M$ is isometric to a solvable Lie group endowed with a left-invariant Riemannian metric.

From now on, we suppose that $M$ is of rank one, so that $M=\F\s{H}^{n}$ is a hyperbolic space over a finite-dimensional
real division algebra $\F\in \{ \R,\C,\H,\O \}$.
We may choose ${G}$ and ${K}$ as in Table~\ref{tab:data-fhn}, so that ${G}$ is (up to a finite cover) the identity component of the isometry group of $M$ and ${K}={G}_o$ is the isotropy subgroup at a point $o\in M$.
Because $M$ is a symmetric space of rank one, any maximal abelian subspace $\g{a}\subseteq \g{p}$ is one-dimensional.

\begin{table}[h!]
\centering
\caption{Data associated with each hyperbolic space.}\label{tab:data-fhn}
\begin{tabular}{*{6}{>{$}l<{$}}}
\toprule
M & {G} & K & {K_0} & \g{g}_\alpha & \g{g}_{2\alpha} \\ 
\midrule[1pt]
\R\s{H}^n & \s{SO}^0(1,n) & \s{SO}(n) & \s{SO}(n-1) & \R^{n-1} & 0 \\
\midrule
\C\s{H}^n & \s{SU}(1,n) & \s{S}(\s{U}(1)\times\s{U}(n)) & \s{S}(\s{U}(1)\times\s{U}(n-1)) & \C^{n-1} & \R \\
\midrule
\H\s{H}^n & \s{Sp}(1,n) & \s{Sp}(1)\times \s{Sp}(n) & \s{Sp}(1) \times \s{Sp}(n-1) & \H^{n-1} & \R^3 \\
\midrule
\O\s{H}^2 & \s{F}_4^{-20} & \s{Spin}(9) & \s{Spin}(7) & \O & \R^7 \\
\bottomrule
\end{tabular}%
\end{table}

The set of (restricted) roots is of the form $\Sigma = \{\pm \alpha, \pm 2\alpha\}$ for some $\alpha \in \g{a}^*$ 
(except if $\F = \R$, when $\Sigma = \{\pm\alpha\}$), and thus the root space decomposition of $\g{g}$ is 
\begin{equation*}
\g{g}=\g{g}_{-2\alpha}\oplus\g{g}_{-\alpha}\oplus\g{g}_{0}\oplus\g{g}_{\alpha}\oplus\g{g}_{2\alpha}.
\end{equation*}
The connected subgroup $K_{0}$ of $G$ whose Lie algebra is $\g{k}_0$ is also presented in Table~\ref{tab:data-fhn}.
We choose a notion of positivity by defining $\alpha$ to be positive (in fact, simple), and the corresponding Iwasawa decomposition becomes $\g{g}=\g{k}\oplus\g{a}\oplus\g{n}$ with $\g{n}=\g{g}_\alpha\oplus\g{g}_{2\alpha}$.

We consider on $\g{g}$ the inner product $\langle \cdot, \cdot \rangle$ given by
\[
\langle X,Y \rangle 
=\frac{2}{(n+3)\dim\F-4}\,\mathcal{B}_\theta(X,Y), \quad X, Y\in \g{g}.
\]
This inner product is taken so that the vector $H_\alpha\in \g{a}$, 
characterized by the equation $\alpha(H)=\langle H_\alpha, H \rangle$, 
has norm $\lvert H_\alpha \rvert = \lvert \alpha \rvert = {1}/{2}$.
We normalize the metric on $M$ so that its restriction to $T_oM\cong \g{p}$ is precisely $\langle \cdot, \cdot \rangle$.
We also define 
\[
B=2H_\alpha,
\] 
which is a unit vector in $\g{a}$.

We use the following two facts throughout this paper.
Let $\lambda\in\Sigma$ and $X,\,Y\in\g{g}_\lambda$.
Then, 
\[
[\theta X,Y]\in\g{k}_0
\text{ if $\langle X,Y\rangle=0$, and }
[\theta X,X]=\lvert X\rvert^2 H_\lambda.
\]

Recall that the solvable group $AN$ acts simply transitively on $M=\mathbb{F}\s{H}^n$.
In particular, $T_o\mathbb{F}\s{H}^n$ is isomorphic to $\g{a}\oplus\g{n}$.
We denote by $\langle \cdot , \cdot \rangle_{AN}$ the pull-back metric induced on $AN$
from the diffeomorphism $M\cong AN$.
Its restriction to $\g{a}\oplus \g{n}$ is determined by the conditions
\[
\begin{aligned}
\langle H_1,H_2\rangle_{AN}
&{}=\langle H_1,H_2\rangle,\
H_1, H_2\in\g{a},
&\langle U,V\rangle_{AN}
&{}=\frac{1}{2}\langle U,V\rangle,\
U, V\in\g{n}.
\end{aligned}
\]

Now, consider the nilpotent algebra $\g{n}=\g{g}_\alpha\oplus \g{g}_{2\alpha}$ endowed with the inner product 
$\langle \cdot,\cdot \rangle_{AN} = \tfrac{1}{2}\langle \cdot,\cdot \rangle$.
For each $Z\in\g{g}_{2\alpha}$, define $J_Z\in \End(\g{g}_\alpha)$ as in Subsection~\ref{sec:heisenberg}.
It is not hard to check that
\begin{equation*}
    J_Z=(\ad(Z)\circ\theta)\vert_{\g{g}_\alpha}
    \qquad
    \text{and}
    \qquad
    J_Z ^2 = - \lvert Z \rvert_{{AN}}^2\id_{\g{g}_\alpha},
\end{equation*}
so $\g{n}$ becomes a generalized Heisenberg algebra.
Moreover, $\g{n}$ satisfies the so-called $J^2$ \textit{condition}:
given an element $V\in \g{g}_\alpha$ and orthogonal vectors $X$, $Y\in \g{g}_{2\alpha}$, there exists a $Z\in \g{g}_{2\alpha}$ (depending on $X$, $Y$ and $V$) satisfying $J_X J_Y V = J_Z V$.
By~\cite[Theorem~1.1]{cowling-h-type}, any generalized Heisenberg algebra that satisfies the $J^2$ condition is isometrically isomorphic to the nilpotent part of the Iwasawa decomposition associated with a rank one symmetric space.

Taking into account the difference between the inner product of $\g{a}\oplus\g{n}$ 
and the inner product of $\g{g}$, 
it is convenient to restate some formulae from Subsection~\ref{sec:heisenberg}:
\begin{align*}
	\langle[U,V],X\rangle
	&{}=\langle J_X U,V\rangle,
	&[U,J_X U]
	&{}=\frac{1}{2}\langle U,U\rangle X,\\
	\langle J_X U,J_X V\rangle
	&{}=\frac{1}{2}\langle U,V\rangle\lvert X\rvert^2,
	&\langle J_X U,J_Y U\rangle
	&{}=\frac{1}{2}\lvert U\rvert^2\langle X,Y\rangle,
\end{align*}
where $U,\,V\in\g{g}_\alpha$, $X,\,Y\in\g{g}_{2\alpha}$.

Using the fact that the projection map $(1-\theta)/2\colon \g{g}_{\alpha}\to \g{p}_{\alpha}$ is a vector space isomorphism,
we can also define a map $J_{X}\in \End(\g{p}_{\alpha})$ for each $X\in \g{g}_{2\alpha}$
by imposing that $J_{X}(1-\theta)U=(1-\theta)J_{X}U$ for each $U\in \g{g}_{\alpha}$.

If $V$ is a vector in either $\g{g}_{\alpha}$ or $\g{p}_{\alpha}$, we define the subspaces
$\g{J}V=\{J_Z V:Z\in\g{g}_{2\alpha}\}$ and $\F V = \R V \oplus \g{J} V$.
Observe that if $V \neq 0$ then $\dim \F V = \dim \F$.

Because the cohomogeneity of the isotropy representation ${K}\curvearrowright \g{p}$ is one, 
we deduce that ${K}$ acts transitively on each sphere of $\g{p}$ centered at the origin.
In particular, two vectors $U$, $V\in\g{p}$ are conjugate under ${K}$ 
if and only if $\lvert U \rvert = \lvert V \rvert$, 
and in that case their corresponding Jacobi operators are conjugate under $\Ad({K})\subseteq \s{O}(\g{p})$.
This means that the Jacobi operators $R_U$, $R_V$ associated with two nonzero vectors $U$, $V\in \g{p}$ have the same eigenvalues (counting multiplicities) up to some positive scalar dependent on the lengths of $U$ and $V$.

Let $R$ denote the $(1,3)$ curvature tensor of $\mathbb{F}\s{H}^n$.
Recall that the Jacobi operator for $U\in\g{p}$ is defined to be
the linear map $R_U\colon\g{p}\to\g{p}$, $V\mapsto R_U(V)=R(V,U)U$.
Since $\mathbb{F}\s{H}^n$ is a symmetric space, the curvature tensor $R$
at $T_o\mathbb{F}\s{H}^n\cong \g{p}$ satisfies $R(X,Y)=-\ad([X,Y])$ for all $X,\,Y\in \g{p}$,
which in turn gives $R_U = -\ad(U)^2$.

We will particularly make use of the Jacobi operators coming from vectors of $\g{a}$ and $\g{p}_\alpha$.
These can be calculated easily from the formulae above, but we include them here for further reference.
We consider $R_{B}$ and $R_{U}$ for $U\in \g{p}_\alpha$ with $\lvert U \rvert=1$.
Then,
\begin{equation}\label{eq:jacobi}
\begin{aligned}
R_{B}V
&{}=
\begin{cases}
0, & V\in \g{a}, \\
-\frac{1}{4} V, & V\in \g{p}_\alpha, \\
-V, & V\in \g{p}_{2 \alpha}, \\
\end{cases}
&
R_{U}V 
&{}= 
\begin{cases}
0, & V\in \R U, \\
-\frac{1}{4}V, & V\in \g{a}\oplus(\g{p}_\alpha\ominus\F U)\oplus\g{p}_{2\alpha}, \\
-V, & V\in \g{J}U.
\end{cases}
\end{aligned}
\end{equation}
It is clear from the above equations that a rank one symmetric space $M\neq \R\s{H}^n$
has negative curvature $-1\leq \sec \leq -{1}/{4}$ and is quarter pinched.
In addition, if $\F=\R$ we deduce that the sectional curvature of $\R\s{H}^n$
is constant and equal to $-1/4$.

If $\F=\H$ we also need a more complete view of $\g{a}\oplus\g{n}$ as a Clifford module
of $\g{g}_{2\alpha}\cong \operatorname{Im}\H$.
Let $X\in\g{g}_{2\alpha}$ be arbitrary.
Taking into account that $\g{a}\oplus\g{n}$ is isomorphic to $T_o\H\s{H}^n$
and $\g{a}\oplus\g{g}_{2\alpha}$ is isomorphic to $\H$,
we can extend the map $J_X\in\End(\g{g}_{\alpha})$ to an endomorphism
$J_X\in\End(\g{a}\oplus\g{n})$.
This map satisfies
\begin{align*}
J_X B&{}=X, 
&J_X X&{}=-\lvert X\rvert_{AN}^2B,
\end{align*}
We also wish to determine $J_X Y$ and $J_X J_Y$ for $X$, $Y\in\g{g}_{2\alpha}$.
This is done by considering a positively oriented orthonormal basis of $\g{g}_{2\alpha}$ and identifying it with the basis of imaginary units of $\mathbb{H}$.
Let $(X_1,X_2,X_3)$ be a positively oriented orthonormal basis of $\g{g}_{2\alpha}\cong\operatorname{Im}\H$; then,
\begin{align*}
J_{X_i}X_{i+1}&{}=-X_{i+2},
&J_{X_i}J_{X_{i+1}}&{}=-J_{X_{i+2}},
\end{align*}
where the indices are taken modulo 3.

In a similar manner, $T_o\mathbb{F}\s{H}^n$ is also isomorphic to $\g{p}$, so all of these considerations can be carried out to $\g{p}$ as well.
The inner product on $\g{p}$ is simply the restriction of the inner product of $\g{g}$ to $\g{p}$.
Using the fact that $(1-\theta)/2\colon\g{a}\oplus\g{n}\to\g{p}$ is an isometry, 
we can decompose $\g{p}=\g{a}\oplus\g{p}_\alpha\oplus\g{p}_{2\alpha}$,
where $\g{p}_\lambda=(1-\theta)\g{g}_\lambda$, $\lambda\in\{\alpha,2\alpha\}$.

For each $X\in\g{g}_{2\alpha}$ we can define operators $J_X\colon\g{p}\to\g{p}$. 
Using the linear isometry $(1-\theta)/2\colon\g{a}\oplus\g{n}\to\g{p}$ again, 
we obtain for all $V\in\g{a}\oplus\g{n}$ that
\[
J_X(1-\theta)V=(1-\theta)J_X V.
\]
In particular, $J_X B=\frac{1}{2}(1-\theta)X$ when we regard $J_X$ as an
operator on $\g{p}$.

With these definitions, a subspace $\g{v}$ of $\g{p}\cong\H^n$ is \emph{totally complex} 
if and only if there exists $X\in\g{g}_{2\alpha}$ such that $J_X(\g{v})=\g{v}$
and $J_Y(\g{v})$ is orthogonal to $\g{v}$ whenever $Y\in\g{g}_{2\alpha}$ is orthogonal to $X$.

\subsection{Totally geodesic submanifolds of symmetric spaces of rank one}\label{sec:fhn-tg-submanifolds}\hfill

In this subsection we recall the classification of complete totally geodesic submanifolds (of dimension greater than one) of the hyperbolic spaces $\F\s{H}^n$ under investigation.
Wolf~\cite{wolf} classified totally geodesic submanifolds of the compact rank one symmetric spaces, and we may apply the duality of symmetric spaces to derive our classification in hyperbolic spaces.
See also~\cite[Section~5]{kollross-octonions} for a detailed treatment of the totally 
geodesic submanifolds of $\O\s{H}^2$.
From this point on, we denote by $\R\s{H}^k(c)$ the real hyperbolic space of constant curvature $-c^{-2}$.
In particular, we have $\R\s{H}^k(2)=\R\s{H}^k$.

Before going through the classification, we recall that every (complete) totally geodesic submanifold $\Sigma$ of a symmetric space of noncompact type $M$ is closed and embedded.
In addition, $\Sigma$ is determined by its tangent space at any point $p \in \Sigma$ because $\Sigma = \exp_p (T_p\Sigma)$.
Furthermore, a subspace $\g{v}\subseteq \g{p}\cong T_o M$ is the tangent space of a totally geodesic submanifold $\Sigma$ of $M$ containing $o=eK$ if and only if $[[\g{v},\g{v}],\g{v}]\subseteq \g{v}$.
The subspaces of $\g{p}$ that satisfy the latter condition are known as Lie triple systems.

The case of $M=\R\s{H}^n$ is by far the simplest.
For each $2\leq k \leq n$ there exists a unique complete $k$-dimensional totally geodesic submanifold of $M$, and it is congruent to $\R\s{H}^k$.
This submanifold can be obtained as the orbit through $o=e{K}$ of the subgroup $\s{SO}^0(1,k)\subseteq \s{SO}^0(1,n)$ embedded in the standard manner.

Let us consider the complex case.
For each $2\leq k\leq n$, the orbits through the origin of the standard subgroups $\s{SO}^0(1,k)$ and $\s{SU}(1,k)$
are totally geodesic $\R\s{H}^k(2)$ and $\C\s{H}^k$, respectively.
One sees that $\R\s{H}^k(2)$ is a totally real submanifold of $\C\s{H}^n$, whereas $\C\s{H}^k$ is a complex submanifold.
We also have a totally geodesic $\C\s{H}^1\cong \R \s{H}^2(1)$ appearing as an orbit of $\s{SU}(1,1)$.
Unlike the submanifolds $\R\s{H}^k(2)$, this $\R\s{H}^2(1)$ is a complex submanifold of $\C\s{H}^n$.
These examples give all totally geodesic submanifolds of $\C\s{H}^n$ up to congruence.

Now we list all totally geodesic submanifolds of $\H\s{H}^n$ up to congruence.
On the one hand, the orbits through $o=e{K}$ of the subgroups $\s{SO}^0(1,k)$, $\s{SU}(1,k)$, and $\s{Sp}(1,k)$ 
($2\leq k\leq n$), all embedded in the standard manner, 
give rise to totally geodesic $\R\s{H}^k(2)$, $\C\s{H}^k$, and $\H\s{H}^k$, respectively.
In particular, the submanifold $\R\s{H}^k(2)$ is totally real,
$\C\s{H}^k$ is totally complex, and $\H\s{H}^k$ is quaternionic.
On the other hand, the orbit $\s{Sp}(1,1)\cdot o$ is a totally geodesic $\H\s{H}^1\cong \R \s{H}^4(1)$ of constant curvature equal to $-1$.
Because $\R\s{H}^4(1)$ has constant curvature, any $k$-dimensional real subspace of $T_o\R\s{H}^4(1)$ exponentiates to a totally geodesic $\R\s{H}^k(1)$ inside $\R\s{H}^4(1)$.

Finally, we comment on the Cayley case.
The submanifolds $\R\s{H}^2(2)\subseteq \C\s{H}^2\subseteq \H\s{H}^2$ appear as the orbits of the subgroups $\s{SO}^0(1,2)\subseteq \s{SU}(1,2)\subseteq \s{Sp}(1,2)$ at $o$.
The Lie algebra embedding $\g{sp}(1,2)\hookrightarrow\g{f}_4^{-20}$ is given explicitly in~\cite[Proposition~5.3]{kollross-octonions}.
In addition, we have a totally geodesic $\O\s{H}^1\cong \R\s{H}^8(1)$ arising as the orbit of the action of an $\s{SO}^0(1,8)\subseteq \s{F}_4^{-20}$.
Similarly to the quaternionic case, the fact that $\O\s{H}^1$ has constant curvature means that every subspace of $T_o\O\s{H}^1$ exponentiates to a totally geodesic submanifold isometric to $\R\s{H}^k$.
Altogether these examples yield all totally geodesic submanifolds of $\O\s{H}^2$
up to congruence.

\subsection{Polar actions}\label{sec:polar}\hfill

We now recall the basic notions concerning polar actions on Riemannian manifolds.
Suppose $M$ is a complete Riemannian manifold and $H$ is a Lie group acting isometrically on $M$.
For the purposes of this article, we may assume that $H$ acts effectively on $M$, which means that the only element of $H$ that acts trivially on $M$ is the identity element.
Hence, we can view $H$ as a subgroup of the isometry group $I(M)$, and if $H$ is connected, then it corresponds to a subgroup of the identity component $I^0(M)$.

The action of $H\subseteq I(M)$ on $M$ is called proper if the shear map $H\times M \to M \times M$ defined by $(h,p)\mapsto (h\cdot p, p)$ is a proper map.
This condition is equivalent to $H$ being a closed subgroup of $I(M)$.
If the action of $H$ is proper, then one sees that the isotropy subgroups of the action are compact, the $H$-orbits are closed and embedded in $M$, and the orbit space $M/H$ is Hausdorff when endowed with the quotient topology.

The orbits of a proper isometric action can be classified according to their type.
We say that two orbits $H\cdot p$ and $H\cdot q$ have the same type if the isotropy subgroups $H_p$ and $H_q$ are conjugate.
The orbit type of the orbit $H \cdot p$ is denoted by $[H\cdot p]$.
In addition, given two orbits $H\cdot x,\,H\cdot y$, we write $[H\cdot x]\leq [H\cdot y]$ if $H_y$ is conjugate to a subgroup of $H_x$.
The relation $\leq$ on the set of orbit types of the $H$-action can be shown to be a partial ordering and it possesses a maximum element.
The $H$-orbits whose orbit type is maximum with respect to $\leq$ are called principal orbits,
and their union constitutes a dense open subset of $M$.
A nonprincipal orbit $H\cdot p$ is called singular if its dimension is strictly smaller than that of the principal orbits,
and exceptional if it has the same dimension as the principal orbits.

In this work, we are interested in distinguishing isometric actions in terms of their orbits.
To this effect, we say that two isometric actions $H_1 \curvearrowright M$ and $H_2 \curvearrowright M$ are orbit equivalent if there exists an isometry of $M$ that carries the orbits of the $H_1$-action to the orbits of the $H_2$-action.

Given a proper isometric action of $H\subseteq I(M)$ on $M$, a complete and injectively immersed submanifold $\Sigma \subseteq M$ is called a section of the action if
\begin{enumerate}[\rm (i)]
    \item $\Sigma$ meets all orbits of $H$, and
    \item For every point $p\in \Sigma$, the tangent spaces $T_p(H\cdot p)$ and $T_p \Sigma$ are orthogonal.
\end{enumerate}
The action of $H$ is called polar if it admits a section.
It can be shown that if $\Sigma$ is a section of the $H$-action,
then $\Sigma$ is totally geodesic in $M$ and any other section of the action
is congruent to $\Sigma$ under an element of $H$.
If one (and thus every) section is flat with the induced metric, then we say that the action of $H$ is hyperpolar.

Podestà and Thorbergsson observed in~\cite[Lemma~1A.4]{PodestaThorbergsson99} that if $H$ is a compact Lie group acting polarly on a simply connected symmetric space $M$ with singular orbits, then any section of the action $H\curvearrowright M$ contains a totally geodesic hypersurface.
In particular, if $M$ is a symmetric space of compact type and rank one, then it is shown from this and the fact that any polar action on $M$ has at least one singular orbit~\cite[Lemma~1A.2]{PodestaThorbergsson99} that any section of the action $H\curvearrowright M$ has constant sectional curvature.
The argument by Podestà and Thorbergsson also applies in the noncompact case.
This was noted by Kollross when $M$ is the Cayley hyperbolic plane~\cite[Lemma~6.1]{kollross-octonions}, but the proof carries over to the general case.

\begin{lemma}\label{lemma:singular}
    Let $M$ be a Riemannian manifold and $H$ a Lie group acting nontrivially and polarly on $M$ with section $\Sigma$.
    If the action of $H$ has a singular orbit, then $\Sigma$ contains a totally geodesic hypersurface.
    In particular, if $M$ is a symmetric space of noncompact type and rank one, then $\Sigma$ is isometric to $\R\s{H}^{k}(c)$, where $c\in \{1,2\}$.
\end{lemma}

\begin{proof}
    Let $p\in \Sigma$ be a point such that $H\cdot p$ is singular.
    Such a point exists because $\Sigma$ meets all orbits.
    The slice representation $H_p \curvearrowright \nu_p (H\cdot p)$ defined 
    by the equation $h\cdot\xi=h_{*p}\xi$
    is a nontrivial polar representation with section $T_p\Sigma$, see~\cite{BerndtConsoleOlmos16}.

    We consider the normalizer and centralizer of $T_p \Sigma$ in $H_p$, which are defined respectively as follows:
    \begin{equation*}
        N_{H_p}(T_p\Sigma)=\{ g\in H_p \colon g(T_p\Sigma)=T_p\Sigma \}, \quad
        Z_{H_p}(T_p\Sigma)=\{ g\in H_p \colon g\vert_{T_p\Sigma}=\operatorname{id}_{T_p\Sigma} \}.
    \end{equation*}
    Then, $N_{H_p}(T_p\Sigma)$ is a closed subgroup of $H_p$, $Z_{H_p}(T_p\Sigma)$ is an open normal subgroup of $N_{H_p}(T_p\Sigma)$, and the quotient $\Pi(T_p\Sigma)=N_{H_p}(T_p\Sigma)/Z_{H_p}(T_p\Sigma)$ is a finite group, known as the polar group of the slice representation $H_p \curvearrowright \nu_p(H\cdot p)$
    (see for instance~\cite{GroveZiller}).
    Because the orbit $H\cdot p$ is singular, the polar group $\Pi(T_p\Sigma)$ is nontrivial and generated by reflections along hyperplanes of $T_p\Sigma$.

    Let $g\in N_{H_p}(T_p\Sigma)$ be an element such that $g_{*p}\in \s{O}(T_p\Sigma)$ is the reflection along a hyperplane $V\subseteq T_p\Sigma$.
    Then $g\cdot \Sigma = \Sigma$ and the restriction $g\colon \Sigma \to \Sigma$ is an involutive isometry.
    Denote by $\mathcal{H}\subseteq \Sigma$ the connected component of $\operatorname{Fix}(g)$ containing $p$.
    We have that $\mathcal{H}$ is a totally geodesic hypersurface of $\Sigma$ whose tangent space at $p$ is $V$.

    Now, suppose that $M$ is a symmetric space of noncompact type and rank one.
    From the classification of totally geodesic submanifolds of hyperbolic spaces (see Section~\ref{sec:fhn-tg-submanifolds}), we see that the only totally geodesic submanifolds of $M$ admitting a totally geodesic hypersurfaces are those of constant curvature, so the second claim follows. \qedhere
\end{proof}

It is worth mentioning that if $H$ is a Lie group acting nontrivially, nontransitively and polarly on $\C\s{H}^n$, then the section of the $H$-action is totally real and isometric to $\R\s{H}^{k}(2)$, see~\cite[Proposition~2.4]{DiazDominguezKollross20}.
As far as we are aware, in all known examples of polar actions with singular orbits on $\H\s{H}^n$ and $\O\s{H}^2$ the section is a totally geodesic $\R\s{H}^k(2)$~\cite{kollross-duality,kollross-octonions}.
Therefore, it is natural to conjecture whether the value $c=1$ can be omitted from the statement of Lemma~\ref{lemma:singular}.
Moreover, a consequence of Theorem~\ref{th:main} is that every polar action without singular orbits on a hyperbolic space $M\neq\R\s{H}^n$ has a section isometric to $\R\s{H}^k(2)$.
This will be proved in Proposition~\ref{prop:examples}.

We say that the action of $H$ on $M$ induces a foliation if it has no singular orbits.
In the case that $M$ is a symmetric space of noncompact type
(or more generally a Hadamard manifold)
it follows from~\cite[Proposition~2.1]{BerndtDiazTamaru10} that all the orbits of $H$ are principal,
so all isotropy subgroups are conjugate in $H$.

From~\cite[Theorem~4.1]{BerndtDiazTamaru10} we have the following criterion of polarity (see also~\cite{Gorodski04} for the first criterion of polarity):

\begin{proposition}\label{th:criterion}
	Let $M=G/K$ be a Riemannian symmetric space of noncompact type with Cartan decomposition $\g{g}=\g{k}\oplus\g{p}$.
	Let $H\subseteq G$ be a closed subgroup acting on $M$ in such a way that its orbits form a foliation on $M$.
	We define
	\[
	\g{h}_\g{p}^\perp=\big\{\xi\in\g{p}:\langle\xi,X\rangle=0,\text{ for all $X\in\g{h}$}\big\}.
	\]
	Then:
	\begin{enumerate}[\rm (i)]
		\item The group $H$ acts polarly on $M$ if and only if $\g{h}_\g{p}^\perp$ is a Lie triple system in $\g{p}$ and $[\g{h}_\g{p}^\perp,\g{h}_\g{p}^\perp]$ is orthogonal to $\g{h}$.
		\item The group $H$ acts hyperpolarly on $M$ if and only if $\g{h}_\g{p}^\perp$ is an abelian subspace of $\g{p}$.
	\end{enumerate}
	Moreover, if $H_\g{p}^\perp$ denotes the subgroup of $G$ whose Lie algebra is $[\g{h}_\g{p}^\perp,\g{h}_\g{p}^\perp]\oplus\g{h}_{\g{p}}^{\perp}$, then $H_\g{p}^\perp\cdot o$ is a section of the action of $H$ on $M$.
\end{proposition}

\section{The examples}\label{sec:examples}

In this section we study the geometry of the orbits of the subgroups $\s{S}_{\g{b},\g{v}}$
appearing in Theorem~\ref{th:main}.
In particular, we prove Theorem~\ref{th:main}~(\ref{th:main-congruence}).

Let $M=\mathbb{F}\s{H}^n=G/K$ be a symmetric space of rank one and noncompact type, 
where $G$ and $K$ are chosen according to Table~\ref{tab:data-fhn}.
We choose a Cartan decomposition $\g{g}=\g{k}\oplus\g{p}$ of $\g{g}$, the Lie algebra of $G$, 
a $1$-dimensional subspace $\g{a}\subseteq\g{p}$, and 
$\g{g}=\g{g}_{-2\alpha}\oplus\g{g}_{-\alpha}\oplus\g{g}_0
\oplus\g{g}_{\alpha}\oplus\g{g}_{2\alpha}$
the restricted root space decomposition with respect to $\g{a}$.
We write $\g{g}_0=\g{k}_0\oplus\g{a}$ as usual.
We have a unique simple root $\alpha$ and the positive roots are $\alpha,\,2\alpha$ (except then $\F=\R$, where $2\alpha$ is not a root).

Let $\g{b}$ a subspace of $\g{a}$, and $\g{v}$ a totally real subspace of $\g{g}_\alpha$.
We define the vector subspace $\g{s}_{\g{b},\g{v}}=(\g{a}\ominus\g{b})\oplus(\g{g}_\alpha\ominus\g{v})\oplus\g{g}_{2\alpha}\subseteq \g{a}\oplus \g{n}$,
which is a Lie subalgebra of $\g{a}\oplus \g{n}$,
and we let $\s{S}_{\g{b},\g{v}}$ the connected subgroup of $G$ whose Lie algebra is $\g{s}_{\g{b},\g{v}}$.
Because the Lie exponential map $\Exp\colon \g{a}\oplus\g{n}\to AN$ is a diffeomorphism, we have that
$\s{S}_{\g{b},\g{v}}=\Exp(\g{s}_{\g{b},\g{v}})$ is a closed subgroup of $AN$ (and thus of $G$).

Our first task is to prove Theorem~\ref{th:main}~(\ref{th:main-examples}).

\begin{proposition}\label{prop:examples}
The group $\s{S}_{\g{b},\g{v}}$ acts polarly on $M$ inducing a foliation.
In addition,
if and $k=\dim \g{b}+\dim\g{v}\geq 2$, then
every section of the action of $\s{S}_{\g{b},\g{v}}$ is isometric to $\R\s{H}^{k}(2)$.
\end{proposition}

\begin{proof}
Since $\s{S}_{\g{b},\g{v}}$ is a subgroup of $AN$, 
and the action of $AN$ on $M$ is simply transitive,
it is clear that $\s{S}_{\g{b},\g{v}}$ induces a foliation on $M$.
We just have to use the criterion of polarity given in Proposition~\ref{th:criterion} to check that $\s{S}_{\g{b},\g{v}}$ acts polarly on $M$.

We have $(\g{s}_{\g{b},\g{v}})_\g{p}^\perp=\g{b}\oplus(1-\theta)\g{v}$.
The fact that $\g{v}$ is totally real in $\g{g}_\alpha$ implies that 
$(\g{s}_{\g{b},\g{v}})_\g{p}^\perp$ is a Lie triple system in $\g{p}$ and $[\g{v},\g{v}]=0$.
This leads to
$[(\g{s}_{\g{b},\g{v}})_\g{p}^\perp,(\g{s}_{\g{b},\g{v}})_\g{p}^\perp]
\subseteq(1+\theta)(\g{v}\oplus[\theta\g{v},\g{v}])\subseteq(1+\theta)\g{v}\oplus\g{k}_0$, which is orthogonal to $\g{s}_{\g{b},\g{v}}$, as we wanted to check.

We now prove the second assertion.
In fact, because $(1-\theta)\g{v}$ is contained in $\g{a}\oplus (1-\theta)\g{v}$, it suffices to show that $\g{a}\oplus (1-\theta)\g{v}$ is the Lie triple system corresponding to a real hyperbolic space $\R\s{H}^{m}(2)$, where $m=1+\dim \g{v}\geq 2$.
Let $\Sigma$ be the unique totally geodesic submanifold of $M$ with $o\in \Sigma$ and $T_{o}\Sigma=\g{a}\oplus(1-\theta)\g{v}$.
Then $\Sigma$ is automatically a simply connected symmetric space of rank one,
so we can determine it by choosing any unit vector in $T\Sigma$ and computing its Jacobi operator.
For instance, if we take $B\in T_{o}\Sigma$ it is straightforward to check that
\begin{equation*}
R_{B}\xi
=
-\ad(B)^2\xi
=
\begin{cases}
    0, & \xi\in \g{a}, \\
    -\frac{1}{4}\xi, & \xi\in(1-\theta)\g{v}.
\end{cases}
\end{equation*}
By looking at~\eqref{eq:jacobi}, it is clear that $R_B$ is the Jacobi operator of a real hyperbolic space whose curvature is equal to $-\frac{1}{4}$, so necessarily $\Sigma=\R\s{H}^{m}(2)$, as desired.
\end{proof}

Next we deal with the congruence problem of the actions under investigation. 
We first need

\begin{lemma}\label{lemma:action-on-abelian}
Let $\g{v}$ be an abelian subspace of $\g{g}_\alpha$ and $\xi\in\g{v}$.
Then, there exists $k\in K_0$ such that $\Ad(k)\g{v}=\g{v}$ and $\Ad(k)(\xi)=-\xi$.
\end{lemma}

\begin{proof}
Recall that $\g{g}_\alpha\cong\F^{n-1}$ as an $\F$-module, 
and that $K_0$ acts on $\g{g}_\alpha$ as its standard representation if $\F\in\{\R,\C,\H\}$,
or as the unique spin representation $\s{Spin}(7)\to\s{SO}(8)$ if $\F=\O$ (see Table~\ref{tab:data-fhn}).

If $\F=\O$, then $\g{v}$ is $1$-dimensional and 
$\s{Spin}(7)$ acts transitively on the unit sphere of $\g{g}_\alpha$~\cite{Borel50},
so we may assume $\F\in\{\R,\C,\H\}$. 
By linearity we may also suppose that $\lvert\xi\rvert=1$.
Let $\{\xi,e_2,\dots,e_k\}$ be an orthonormal basis of $\g{v}$.
Since $\g{v}$ is abelian, it is also totally real in $\g{g}_\alpha$.
Thus, the elements of this basis are not only orthonormal, but $\F$-orthonormal.
We complete this basis to an $\F$-orthonormal basis 
$\{\xi,e_2,\dots,e_k,e_{k+1},\dots,e_{n-1}\}$ of $\g{g}_\alpha$.
In particular, $\{-\xi,e_2,\dots,e_{n-1}\}$ is another $\F$-orthonormal basis of $\g{g}_\alpha$.
Since $K_0$ acts transitively on the set of all ordered $\F$-orthonormal bases of $\g{g}_\alpha$,
we deduce that there exists $k\in K_0$ such that $\Ad(k)\xi=-\xi$ and $\Ad(k)e_i=e_i$ for all $i$.
Because of this, $\Ad(k)\g{v}=\g{v}$ as well, so the result is proved. \qedhere
\end{proof}

The proof of Theorem~\ref{th:main}~(\ref{th:main-congruence}) follows readily from the next result.

\begin{proposition}
Let $\g{b}_1$ and $\g{b}_2$ be subspaces of $\g{a}$, 
and $\g{v}_1$ and $\g{v}_2$ totally real subspaces of $\g{g}_\alpha$.
Then, the actions of $\s{S}_{\g{b}_1,\g{v}_1}$ and $\s{S}_{\g{b}_2,\g{v}_2}$ on $\F\s{H}^n$ 
are orbit equivalent 
if and only if $\dim\g{b}_1=\dim\g{b}_2$ and $\dim\g{v}_1=\dim\g{v}_2$.
\end{proposition}

\begin{proof}
Since the group $K_0$ acts transitively on abelian subspaces of the same dimension of $\g{g}_\alpha$ 
(see for example~\cite[Lemma~3.2]{DiazLorenzo}),
$\s{S}_{\g{b}_1,\g{v}_1}$ and $\s{S}_{\g{b}_2,\g{v}_2}$ are conjugate by an element of $K_0$ if
$\dim\g{b}_1=\dim\g{b}_2$ and $\dim\g{v}_1=\dim\g{v}_2$.
In particular, their actions are orbit equivalent.

For the converse, let $\mathcal{S}_\xi$ denote the shape operator of a subgroup $S$ of $AN$ 
with respect to a normal vector $\xi$.
Using the Koszul formula for left-invariant vectors of $AN$, it is easy to see that 
the shape operator satisfies
\begin{equation*}
    \langle \mathcal{S}_\xi X, X \rangle_{AN}
    =
    -\langle X,[X,\xi] \rangle_{AN}
    =
    -\frac{1}{2}\langle X,[X,\xi] \rangle
    =
    \frac{1}{2} \langle [\theta X,X],\xi \rangle
\end{equation*}
for all $X\in\g{s}$ and $\xi\in(\g{a}\oplus\g{n})\ominus\g{s}$.

If $\g{b}_i=0$ we consider the orbit through the origin $\s{S}_{0,\g{v}_i}\cdot o$.
As a submanifold of $AN$, its normal space at $o$ can be identified with $\g{v}_i$.
Let $\xi\in\g{v}_i$.
By Lemma~\ref{lemma:action-on-abelian} there exists $k\in K_0$ such that $\Ad(k)$ leaves $\g{v}_i$ invariant 
and maps $\xi$ to $-\xi$.
In particular, $\Ad(k)$ leaves $\g{s}_{0,\g{v}_i}$ invariant, 
so $k$ maps $\s{S}_{0,\g{v}_i}\cdot o$ to itself
sending $\xi$ to $-\xi$.
Since $k$ is an extrinsic isometry, 
$-\mathcal{S}_\xi=\mathcal{S}_{k_*\xi}=k_*\mathcal{S}_\xi k_*^{-1}$,
which implies $\tr\mathcal{S}_\xi=0$.
As $\xi\in\g{v}_i$ is arbitrary,
we conclude that $\s{S}_{0,\g{v}_i}\cdot o$ is minimal.

If $\g{b}_i=\g{a}$ then $\g{s}_{\g{a},\g{v}_i}\subseteq\g{n}$ is a nilpotent ideal of $\g{a}\oplus\g{n}$.
Thus, $\s{S}_{\g{a},\g{v}_i}$ is a nilpotent normal subgroup of $AN$.
This implies that for each $g\in AN$, 
$\s{S}_{\g{a},\g{v}_i}\cdot g(o)=gg^{-1}\s{S}_{\g{a},\g{v}_i}g\cdot o=g(\s{S}_{\g{a},\g{v}_i}\cdot o)$,
so any two orbits of $\s{S}_{\g{a},\g{v}_i}$ are isometrically congruent to each other.
Hence, one orbit is minimal if and only if all of them are.
We calculate the mean curvature vector of $\s{S}_{\g{a},\g{v}_i}\cdot o$ at $o$.
Note that the normal space of $\s{S}_{\g{a},\g{v}_i}\cdot o$ at $o$ can be identified with
$\g{a}\oplus\g{v}_i$.
If $\xi\in\g{v}_i$, the same argument as above implies $\tr\mathcal{S}_\xi=0$, 
so it suffices to calculate $\mathcal{S}_B$.
If $U\in\g{g}_\alpha\ominus\g{v}_i$ we get
\[
\langle\mathcal{S}_B U,U\rangle_{AN}
=\frac{1}{2}\langle[\theta U,U],B\rangle
=\frac{1}{2}\langle \lvert U \rvert^2 H_{\alpha},B\rangle
=\frac{1}{4}\lvert U \rvert^2
=\frac{1}{2}\lvert U\rvert_{AN}^2.
\]
Similarly, $\langle\mathcal{S}_B X,X\rangle_{AN}=\lvert X\rvert_{AN}^2$ for $X\in\g{g}_{2\alpha}$.
Altogether this implies that the mean curvature vector of $\s{S}_{\g{a},\g{v}_i}\cdot o$ at $o$ is
\[
\frac{1}{2}(\dim\g{g}_\alpha-\dim\g{v}+2\dim\g{g}_{2\alpha})B,
\]
which is nonzero.

Therefore, if the actions of $\s{S}_{\g{b}_1,\g{v}_1}$ and $\s{S}_{\g{b}_2,\g{v}_2}$ are orbit equivalent 
we must have $\g{b}_1=\g{b}_2$.
For dimension reasons we also need $\dim\g{v}_1=\dim\g{v}_2$, as claimed.
\end{proof}

\section{General results}\label{sec:general-results}

From now on, our aim is to show that any homogeneous polar foliation on a rank one symmetric space of noncompact type is induced by the action of a group whose action is orbit equivalent to one of the examples given in Theorem~\ref{th:main}~(\ref{th:main-examples}).
We follow the notation from Section~\ref{sec:examples}.

We assume that $H$ is a closed subgroup of $G$ that acts polarly on $M$ inducing a foliation.
The action of $H$ is assumed to be nontrivial and nontransitive.

A Borel subalgebra of a Lie algebra $\g{g}$ is a maximal solvable subalgebra of $\g{g}$.
Borel subalgebras of semisimple Lie algebras have been described in~\cite{Mostow}.

If $\g{g}$ is the Lie algebra of a symmetric space of rank one and noncompact type, 
then the description of its Borel subalgebras is relatively easy.
In this case there exists a Cartan decomposition $\g{g}=\g{k}\oplus\g{p}$ and we have two possibilities:
either it is maximally compact, that is, a maximal abelian subspace of $\g{k}$; 
or maximally noncompact, that is, $\g{t}\oplus\g{a}\oplus\g{n}$, 
where $\g{g}=\g{k}\oplus\g{a}\oplus\g{n}$ is an Iwasawa decomposition and $\g{t}$ is a maximal abelian subspace of
$\g{k}_0=\g{k}\cap\g{g}_0$.

Borel subalgebras are relevant in this context because of the following result~\cite[Proposition~4.2]{DiazLorenzo}:

\begin{proposition}\label{th:Borel:tan}
The leaves of a homogeneous polar foliation on a symmetric space of noncompact type $M=G/K$ coincide, up to isometric congruence, with the orbits of a connected closed solvable subgroup $S$ of $G$ whose Lie algebra $\g{s}$ is contained in a maximally noncompact Borel subalgebra of the form $\g{t}\oplus\g{a}\oplus\g{n}$, where $\g{t}\subseteq \g{k}_{0}$ is abelian.
\end{proposition}

According to Proposition~\ref{th:Borel:tan} above, we may assume from now on that, up to orbit equivalence, the Lie algebra $\g{h}$ of $H$ is contained in a maximally noncompact subalgebra of the form $\g{t}\oplus\g{a}\oplus\g{n}$, with $\g{t}\subseteq \g{k}_0$ an abelian subspace.

As a matter of notation, if $\g{v}$ and $\g{w}$ are vector subspaces of $\g{g}$, 
by $\g{v}_{\g{w}}$ we mean the orthogonal projection of $\g{v}$ onto $\g{w}$.

We recall the result~\cite[Lemma~4.3]{DiazLorenzo}, whose short proof we include here for completeness.

\begin{lemma}\label{lemma:projectionImpliesContained}
Let $\g{q}$ be a Lie subalgebra of $\g{t}\oplus\g{a}\oplus\g{n}$
and $\lambda\in\{\alpha,2\alpha\}$.
If $\g{a}\oplus\g{g}_{\lambda}\subseteq \g{q}_{\g{a}\oplus\g{n}}$, 
then $\g{g}_{\lambda}\subseteq \g{q}$.
\end{lemma}

\begin{proof}
Recall that $\g{q}_{\g{a}\oplus\g{n}}$ denotes the orthogonal projection of $\g{q}$ onto $\g{a}\oplus\g{n}$.
Take $X\in\g{g}_{\lambda}\subseteq\g{q}_{\g{a}\oplus\g{n}}$. 
Then there exist vectors $S$, $T\in\g{t}$ such that $S+B$, $T+X\in\g{q}$. 
In particular, $\ad(S)X+\lambda(B)X=[S+B,T+X]\in\g{q}$. 
This means that the linear map $\ad(S)+\lambda(B)\id_{\g{g}_{\lambda}}$ preserves $\g{g}_{\lambda}$ 
and carries $\g{g}_{\lambda}$ to $\g{q}$. 
Since $S\in\g{t}$, the linear transformation $\ad(S)$ is skew-adjoint, so $\ad(S)+\lambda(B)\id_{\g{g}_{\lambda}}$ is a linear isomorphism and it follows that $\g{g}_{\lambda}\subseteq \g{q}$.
\end{proof}

Following the notation from Proposition~\ref{th:criterion},
we let $\Sigma\subseteq M$ be the section through $o$ of the action, and let $\g{h}_{\g{p}}^\perp$ be its tangent space at $o$.
Before continuing, we prove a result that will be used several times.


\begin{lemma}\label{lemma:equivalence-B+galpha}
Consider the following conditions for a subalgebra $\g{h}\subseteq\g{t}\oplus\g{a}\oplus\g{n}$:
\begin{enumerate}[{\rm (i)}]
\item\label{lemma-item:diagonal-normal} There exists $\xi_\alpha\in\g{g}_\alpha$ such that $B+(1-\theta)\xi_\alpha\in\g{h}_\g{p}^\perp$.
\item\label{lemma-item:normal-a-p1} Either $\g{a}\subseteq\g{h}_\g{p}^\perp$ or $\g{p}_\alpha\cap\g{h}_\g{p}^\perp\neq 0$.
\end{enumerate}
If $\g{h}$ satisfies~\textup{(\ref{lemma-item:diagonal-normal})}, then there exists $g\in N$
such that $\Ad(g)\g{h}\subseteq\g{t}\oplus\g{a}\oplus\g{n}$ satisfies~\textup{(\ref{lemma-item:normal-a-p1})}.
Conversely, if $\g{h}$ satisfies~\textup{(\ref{lemma-item:normal-a-p1})}, then $\g{h}$ is conjugate
under $N$ to a subalgebra of $\g{t}\oplus\g{a}\oplus\g{n}$ satisfying~\textup{(\ref{lemma-item:diagonal-normal})}.
\end{lemma}

Note that, in particular, conditions~(\ref{lemma-item:diagonal-normal})
and~(\ref{lemma-item:normal-a-p1}) are equivalent up to orbit equivalence.

\begin{proof}
Assume that there exists $\xi_\alpha\in\g{g}_\alpha$ such that $B+(1-\theta)\xi_\alpha\in\g{h}_\g{p}^\perp$.
If $\xi_\alpha=0$, this clearly implies that $\g{a}\subseteq\g{h}_\g{p}^\perp$ and it suffices to take $g$ to be the identity element.
Assume $\xi_\alpha\neq 0$.
Then we define $g=\Exp(-2\xi_\alpha/\lvert\xi_\alpha\rvert^2)\in N$ and consider the conjugate subgroup 
$\hat{H}=gHg^{-1}\subseteq {G}$.
The Lie algebra 
$\hat{\g{h}}=\Ad(g)\g{h}$ of $\hat{H}$ is orthogonal to
\begin{align*}
\Ad(g^{-1})^*(B+\xi_\alpha)={}
&e^{\frac{2}{\lvert\xi_\alpha\rvert^2}\ad(\xi_\alpha)^*}(B+\xi_\alpha)
=e^{-\frac{2}{\lvert\xi_\alpha\rvert^2}\ad(\theta\xi_\alpha)}(B+\xi_\alpha) \\
\equiv{}& B+\xi_\alpha-\frac{2\lvert\xi_\alpha\rvert^2}{\lvert\xi_\alpha\rvert^2}H_\alpha=\xi_\alpha \pmod{\theta\g{n}}.
\end{align*}
Since $\theta\g{n}$ is already orthogonal to $\Ad(g)\g{h}$, we obtain that $\xi_\alpha$ is perpendicular to $\Ad(g)\g{h}$.
Hence $(1-\theta)\xi_\alpha\in\hat{\g{h}}_\g{p}^\perp$ and $\g{p}_\alpha\cap\hat{\g{h}}_{\g{p}}^\perp\neq 0$.

Conversely, if $\g{a}\subseteq\g{h}_\g{p}^\perp$, then $B\in\g{h}_\g{p}^\perp$ and the result follows trivially.
Assume $\g{p}_\alpha\cap\g{h}_\g{p}^\perp\neq 0$. 
Then, there is a nonzero vector $\xi_\alpha\in\g{g}_\alpha$ such that $(1-\theta)\xi_\alpha\in\g{h}_\g{p}^\perp$.
We define $g=\Exp(2\xi_\alpha/\lvert\xi_\alpha\rvert^2)\in N$ and consider the conjugate subgroup $\hat{H}=g{H}g^{-1}\subseteq {G}$.
A calculation similar to the one above shows that $B+(1-\theta)\xi_\alpha\in\hat{\g{h}}_\g{p}^\perp$ and the result follows.
\end{proof}

Now we show that indeed one of these two equivalent conditions holds.
This could have been obtained as a consequence of~\cite[Lemma~4.4]{DiazLorenzo}, 
but we can simplify the proof in our case.

\begin{lemma}\label{lemma:B+galpha}
    The action of $H$ is orbit equivalent to the action of a group $\hat{H}$ whose Lie algebra is contained in $\g{t}\oplus\g{a}\oplus\g{n}$ and satisfying the following condition: there is $\xi_\alpha\in\g{g}_\alpha$ such that $B+(1-\theta)\xi_\alpha\in\hat{\g{h}}_\g{p}^\perp$.
\end{lemma}

\begin{proof}
We define the vector subspace $\tilde{\g{h}}=\g{h}+\g{g}_{2\alpha}$, which is also a Lie subalgebra of $\g{g}$.

Assume first that $\tilde{\g{h}}_{\g{a}\oplus\g{n}}=\g{a}\oplus\g{n}$.
According to Lemma~\ref{lemma:projectionImpliesContained}, we get $\g{n}\subseteq\tilde{\g{h}}$.
We show that $\g{n}\subseteq\g{h}$.
Let $X\in\g{g}_{2\alpha}$.
We can write $X=[U,V]$, with $U$, $V\in\g{g}_\alpha$.
Since $\g{n}\subseteq\tilde{\g{h}}=\g{h}+\g{g}_{2\alpha}$, 
there exist $Y$, $Z\in\g{g}_{2\alpha}$ such that $U+Y$, $V+Z\in\g{h}$.
But then, $X=[U,V]=[U+Y,V+Z]\in\g{h}$ and it follows that $\g{g}_{2\alpha}\subseteq\g{h}$.
This readily implies $\g{n}\subseteq\tilde{\g{h}}=\g{h}+\g{g}_{2\alpha}=\g{h}$,
which yields $\g{h}_{\g{p}}^\perp\subseteq \g{a}$.
Since the action of $H$ is not transitive, this gives $\g{a}=\g{h}_\g{p}^\perp$ 
and the result follows in this case.

Assume now that $\tilde{\g{h}}_{\g{a}\oplus\g{n}}\neq\g{a}\oplus\g{n}$.
Thus, there exists a nonzero vector in $\g{a}\oplus\g{n}$ that is orthogonal to both $\g{h}$ and $\g{g}_{2\alpha}$.
We may write this vector as $\xi_0+\xi_\alpha$, where $\xi_0\in \g{a}$ and $\xi_\alpha\in \g{g}_{\alpha}$, and because $\xi_0+\xi_\alpha$ and $\theta\g{n}$ are both orthogonal to $\g{h}$, we obtain that $\xi_0+(1-\theta)\xi_\alpha$ is orthogonal to $\g{h}$.
If $\xi_0\neq 0$ we are done after normalization.
If $\xi_0=0$, we simply apply Lemma~\ref{lemma:equivalence-B+galpha} to finish.
\end{proof}

We end this section by recalling~\cite[Proposition~4.8]{DiazLorenzo} adapted to our particular case:

\begin{proposition}\label{prop:orbit-equivalence}
Let $H$ be a connected closed subgroup of $G$ inducing a homogeneous polar foliation on a rank one symmetric space of noncompact type $M$.
Assume that its Lie algebra is contained in a maximally noncompact Borel subalgebra $\g{t}\oplus\g{a}\oplus\g{n}$,
and $\g{h}_{\g{a}\oplus\g{n}}=\g{z}\oplus(\g{n}\ominus\g{v}_\alpha)$, where $\g{v}_\alpha$ is an abelian subspace of $\g{g}_\alpha$, and $\g{z}$ is a subspace of $\g{a}$.
Let $\tilde{H}$ be the connected Lie subgroup of $G$ whose Lie algebra is $\tilde{\g{h}}=\g{z}\oplus(\g{n}\ominus\g{v}_\alpha)$.
Then, $H$ and $\tilde{H}$ have the same orbits.
\end{proposition}

Although the results presented in this section are valid for all hyperbolic spaces,
in order to prove Theorem~\ref{th:main}~(\ref{th:main-classification}) we will need to consider the cases of $\H\s{H}^n$
and $\O\s{H}^2$ separately.

\section{Polar foliations on quaternionic hyperbolic spaces}\label{sec:quaternionic}

Let $H$ be a Lie group acting polarly on $\H\s{H}^n$ (where $n\geq 2$) inducing a foliation.
It follows from Proposition~\ref{th:Borel:tan} that we may assume (up to orbit equivalence) that $H\subseteq \s{Sp}(1,n)$ is solvable and its Lie algebra $\g{h}$ is contained in a maximally noncompact subalgebra of the form $\g{t}\oplus\g{a}\oplus\g{n}$, where $\g{t}\subseteq \g{k}_0$ is abelian.

\begin{proposition}\label{prop:section-constant-curvature}
Suppose $H$ is a closed connected subgroup of $\s{Sp}(1,n)$ acting polarly on $M=\H\s{H}^n$ in such a way that its orbits form a homogeneous foliation.
Let $\Sigma$ be a section of the action.
Then, either $\Sigma$ is totally real or the action of $H$ is trivial.
\end{proposition}

\begin{proof}
First we show that $\Sigma$ is a space of constant curvature.
From the classification of totally geodesic submanifolds in $\H\s{H}^n$ 
(see Subsection~\ref{sec:fhn-tg-submanifolds}), 
it suffices to show that $\Sigma$ is not isometric to a complex or quaternionic hyperbolic space.

Firstly, let us suppose that $\Sigma = \H\s{H}^k$ for some $k\in \{ 2,\dots,n \}$.
This means that $\Sigma$ is a quaternionic submanifold of $M$.
In addition, if $p\in \Sigma$ is any point, the subspace $T_p(H\cdot p)=\nu_p\Sigma$ is also quaternionic, 
so we see that the orbits of the action $H\curvearrowright M$ are quaternionic.
Because a quaternionic submanifold of a quaternionic Kähler manifold is totally geodesic~\cite[Theorem~5]{Gray}, we deduce that $H\cdot p$ is a totally geodesic submanifold of $M$ for all $p\in M$.
If the action of $H$ is not transitive, 
then any two distinct orbits $H\cdot p$ and $H\cdot q$ are totally geodesic (hence minimal), and~\cite[Corollary~5.2]{AlekseevskyDiScala} 
guarantees that both $H\cdot p$ and $H\cdot q$ consist of one point.
Since $p$ and $q$ are arbitrary, we conclude that the action of $H$ is trivial.

Now, assume that $\Sigma = \C\s{H}^k$ for some $2\leq k \leq n$.
Additionally, if the action of $H$ is not transitive, by lemmas~\ref{lemma:equivalence-B+galpha} and~\ref{lemma:B+galpha}, 
we can assume, up to orbit equivalence, 
that $\g{a}\subseteq\g{h}_\g{p}^\perp$ or $\g{p}_\alpha\cap\g{h}_\g{p}^\perp\neq 0$.

Suppose $\g{a}\subseteq \g{h}_{\g{p}}^\perp$.
Then the restriction of the Jacobi operator $R_{B}$ to $\g{h}_{\g{p}}^\perp$ has the same spectral decomposition as that of the Jacobi operator of a unit vector in $T\C\s{H}^k$.
In other words, there exists a nonzero vector $X\in\g{g}_{2\alpha}$ and a subspace $\g{v}\subseteq \g{g}_\alpha$ such that $\g{h}_{\g{p}}^\perp=\g{a}\oplus (1-\theta)\g{v}\oplus \R (1-\theta)X$.
As the action of $H$ is polar, 
we have for every $\xi\in \g{v}$ that 
$(1+\theta)J_X\xi =-(1+\theta)[\theta\xi,X]=[(1-\theta)\xi,(1-\theta)X]$ is orthogonal to $\g{h}$.
This implies that $\g{v}$ is invariant under $J_X$, 
and so is the orthogonal complement $\g{g}_\alpha\ominus \g{v}$.
As the dimension of $\Sigma$ is at most $2n\leq\dim \g{g}_\alpha$, we have $\g{v}\neq \g{g}_\alpha$.
Choose a nonzero vector $V \in \g{g}_\alpha$ that is orthogonal to $\g{v}$.
We have $V$, $J_X V\in \g{h}_{\g{a}\oplus\g{n}}$, 
so we may choose elements $T$, $T'\in\g{t}$ such that $T+V$, $T'+J_X V\in \g{h}$.
Then
\[
[T,J_X V]+[V,T']+\frac{1}{2}\lvert V \rvert^2 X=[T+V,T'+J_X V]\in\g{h}.
\]
This readily gives a contradiction with the fact that $(1-\theta)X$ is orthogonal to $\g{h}$.

Now assume $\g{h}_{\g{p}}^\perp\cap\g{p}_\alpha\neq 0$.
Thus, we consider a vector $\xi_\alpha \in \g{g}_\alpha$ such that $\lvert \xi_\alpha \rvert = 1$ and
$(1-\theta)\xi_\alpha\in \g{h}_{\g{p}}^\perp$.
The restriction of the Jacobi operator $R_{(1-\theta)\xi_\alpha}$ to $\g{h}_{\g{p}}^\perp$ has the same eigenvalues (with multiplicities) as the Jacobi operator of a unit vector in $T\C\s{H}^k$, and this yields $\g{h}_{\g{p}}^\perp=\R (1-\theta) \xi_\alpha \oplus \R (1-\theta)J_Z \xi_\alpha \oplus \g{w}$, 
where $Z\in \g{g}_{2\alpha}$ is a nonzero vector with $\lvert Z\rvert^2=2$ (thus, $\lvert Z\rvert_{AN}=1$)
and $\g{w}$ is a real subspace of $\g{a}\oplus(\g{p}_\alpha\ominus (1-\theta)\H\xi_\alpha)\oplus\g{p}_{2\alpha}$.
Choose vectors $X,\,Y\in\g{g}_{2\alpha}$ such that $(Y,X,Z)$ is a positively oriented orthonormal basis of $\g{g}_{2\alpha}$.
This implies that $J_X J_Y = J_Z$.
In particular, the vectors $J_X \xi_\alpha$ and $J_Y \xi_\alpha$ are perpendicular to 
$\g{h}_{\g{p}}^\perp$, which means that $J_X \xi_\alpha$, $J_Y \xi_\alpha$ are in $\g{h}_{\g{a}\oplus\g{n}}$.
Choose elements $T_X$, $T_Y\in \g{t}$ satisfying $T_X+J_X \xi_\alpha$, $T_Y + J_Y \xi_\alpha\in \g{h}$.
Then we have 
$[T_X,J_Y\xi_\alpha]+[J_X \xi_\alpha,T_Y]+[J_X\xi_\alpha,J_Y\xi_\alpha]
=[T_X+J_X \xi_\alpha,T_Y+J_Y \xi_\alpha]\in \g{h}$.
Observe that
\begin{equation}\label{eq:bracket-of-js}
\begin{aligned}%
{}[J_X\xi_\alpha,J_Y\xi_\alpha]
&{}=-\frac{2}{\lvert X \rvert^2}[J_X\xi_\alpha,J_X^2J_Y\xi_\alpha]\\
&{}=-\frac{2}{\lvert X \rvert^2}[J_X\xi_\alpha,J_XJ_Z\xi_\alpha] \\
&{}={}\frac{2}{\lvert X \rvert^2}[J_X \xi_\alpha,J_Z J_X \xi_\alpha] 
= \frac{\lvert J_X \xi_\alpha \rvert^2}{\lvert X \rvert^2} Z=\frac{1}{2}Z,
\end{aligned}
\end{equation}
so $[T_X,J_Y\xi_\alpha]+[J_X \xi_\alpha,T_Y]+\tfrac{1}{2}Z\in\g{h}$.
Moreover, as the action of $H$ is polar, the vector
$(1+\theta)\bigl(\tfrac{1}{2}Z-[\theta\xi_\alpha,J_Z\xi_\alpha]\bigr)
=[(1-\theta)\xi_\alpha,(1-\theta)J_Z \xi_\alpha]$ is perpendicular to $\g{h}$.
Thus,
\begin{equation*}
0= \Bigl< [T_X,J_Y\xi_\alpha]+[J_X \xi_\alpha,T_Y]+\frac{1}{2}Z,\,
(1+\theta)\Bigl(\frac{1}{2}Z-[\theta\xi_\alpha,J_Z\xi]\Bigr) \Bigr>
=\frac{1}{4}\lvert Z \rvert^2,
\end{equation*}
contradicting the fact that $Z$ is nonzero.

Therefore, we have shown that $\Sigma$ is a space of constant curvature.
Assume that $\Sigma$ is not totally real in $\mathbb{H}\s{H}^n$.
Since $\Sigma$ is a space of constant curvature, 
it must be strictly contained in a totally geodesic $\mathbb{H}\s{H}^1\subseteq\mathbb{H}\s{H}^n$.

By lemmas~\ref{lemma:equivalence-B+galpha} and~\ref{lemma:B+galpha} we have, up to orbit equivalence, 
that $\g{a}\subseteq\g{h}_\g{p}^\perp$ or $\g{p}_\alpha\cap\g{h}_\g{p}^\perp\neq 0$.

If $\g{a}\subseteq\g{h}_\g{p}^\perp$, 
then $\mathbb{H}\g{a}=\g{a}\oplus\g{p}_{2\alpha}$.
Thus, $\g{h}_\g{p}^\perp\subseteq\g{a}\oplus\g{p}_{2\alpha}$.
Take nonzero $U\in\g{g}_\alpha$ and $Z\in\g{g}_{2\alpha}$.
Then there exists $S$, $T\in\g{t}$ such that $S+U$, $T+J_Z U\in\g{h}$.
By Proposition~\ref{th:criterion} we have
\[
0=\langle[S+U,T+J_Z U],(1-\theta)Z\rangle
=\langle [U,J_Z U],Z\rangle
=\frac{1}{2}\lvert U\rvert^2\lvert Z\rvert^2,
\]
which is a contradiction.

If $\g{p}_\alpha\cap\g{h}_\g{p}^\perp\neq 0$, there is a nonzero vector $\xi_\alpha\in\g{g}_\alpha$ such that
$(1-\theta)\xi_\alpha\in\g{h}_\g{p}^\perp$.
Thus, $\g{h}_\g{p}^\perp\subseteq\mathbb{H}(1-\theta)\xi_\alpha$.
In particular $\g{a}\oplus\g{g}_{2\alpha}$ is orthogonal to $\g{h}_\g{p}^\perp$, so $\g{a}\oplus\g{g}_{2\alpha}\subseteq\g{h}_{\g{a}\oplus\g{n}}$.
By Lemma~\ref{lemma:projectionImpliesContained} we get $\g{g}_{2\alpha}\subseteq\g{h}$.
Since $\g{h}_\g{p}^\perp$ is not totally real, 
there exists $\eta_\alpha\in\g{h}_\g{p}^\perp\subseteq\mathbb{H}(1-\theta)\xi_\alpha$ 
such that $[\xi_\alpha,\eta_\alpha]\neq 0$.
By Proposition~\ref{th:criterion} we get
\begin{align*}
0
&{}=\langle[\xi_\alpha,\eta_\alpha],[(1-\theta)\xi_\alpha,(1-\theta)\eta_\alpha]\rangle
=\langle[\xi_\alpha,\eta_\alpha],
(1+\theta)([\xi_\alpha,\eta_\alpha]-[\theta\xi_\alpha,\eta_\alpha])\rangle
=\lvert[\xi_\alpha,\eta_\alpha]\rvert^2,
\end{align*}
which is also a contradiction.

Therefore, $\Sigma$ cannot be a totally geodesic submanifold of 
$\H\s{H}^1\subseteq\H\s{H}^n$.
Since $\Sigma$ has constant curvature, we conclude that $\Sigma$ is totally real.
\end{proof}

Due to Proposition~\ref{prop:section-constant-curvature}, we know that the sections of the action
of $H$ are totally real submanifolds of constant curvature.
This implies that $\g{h}_{\g{p}}^\perp$ is totally real in $\g{p}$.

\begin{lemma}\label{lemma:vector-g2alpha}
$\g{h}\cap\g{g}_{2\alpha}\neq 0$.
\end{lemma}

\begin{proof}
Assume, on the contrary, that $\g{h}\cap\g{g}_{2\alpha}=0$.

Since $\g{h}_{\g{p}}^\perp$ is totally real we have $\dim\g{h}_{\g{p}}^\perp\leq n$.
In this proof, for $\g{v}\subseteq\g{t}\oplus\g{a}\oplus\g{n}$ 
we denote by $\pi_{\g{v}}\colon\g{h}\to\g{v}$ 
the orthogonal projection of $\g{h}$ onto $\g{v}$. 
Recall that $\g{h}_{\g{v}}=\pi_{\g{v}}(\g{h})$.

Since $H$ induces a foliation on $\H\s{H}^n$ and $\g{h}_{\g{p}}^\perp$ is the tangent space of a section, 
we have $\dim\g{h}_{\g{a}\oplus\g{n}}=\dim\H\s{H}^n-\dim\g{h}_{\g{p}}^\perp\geq 3n$.

Now we consider $\pi_{\g{t}}$ the orthogonal projection onto $\g{t}$.
We have $\ker\pi_{\g{t}}=\g{h}\cap(\g{a}\oplus\g{n})$, which is an ideal of $\g{h}$.
Recall that $\g{t}$ is the Lie algebra of a maximal torus of $K_0\cong\s{Sp}(1)\times\s{Sp}(n-1)$
(see Table~\ref{tab:data-fhn}).
In particular, $\dim\g{t}\leq n$.
Taking this into account yields 
$\dim(\g{h}\cap(\g{a}\oplus\g{n}))=\dim\g{h}-\dim\pi_{\g{t}}(\g{h})\geq 3n-n=2n$.

Note that $\g{h}\cap\g{n}$ is also an ideal of $\g{h}$ 
because $\g{h}\cap(\g{a}\oplus\g{n})$ is and
$[\g{t}\oplus\g{a}\oplus\g{n},\g{n}]\subseteq\g{n}$.
Since $\dim\g{a}=1$, 
we have $\dim(\g{h}\cap\g{n})\geq\dim(\g{h}\cap(\g{a}\oplus\g{n}))-1\geq 2n-1$.

Since $[\g{n},\g{n}]\subseteq\g{g}_{2\alpha}$ and $\g{h}\cap\g{n}$ is a subalgebra of $\g{h}$, 
we obtain $[\g{h}\cap\g{n},\g{h}\cap\g{n}]\subseteq\g{h}\cap\g{n}\cap\g{g}_{2\alpha}=\g{h}\cap\g{g}_{2\alpha}=0$ 
by our assumption.
Thus, $\g{h}\cap\g{n}$ is abelian.
But $[\g{n},\g{g}_{2\alpha}]=0$, so $(\g{h}\cap\g{n})_{\g{g}_\alpha}$ is also abelian.
Since abelian subspaces of $\g{g}_\alpha$ are precisely totally real subspaces, we conclude that $\dim(\g{h}\cap\g{n})_{\g{g}_\alpha}\leq \dim_{\H}\g{g}_\alpha=n-1$.

On the other hand, the kernel of the orthogonal projection of 
$\g{h}\cap\g{n}$ onto $\g{g}_\alpha$ is precisely $\g{h}\cap\g{n}\cap\g{g}_{2\alpha}=\g{h}\cap\g{g}_{2\alpha}=0$.
Thus, $\dim(\g{h}\cap\g{n})_{\g{g}_\alpha}=\dim(\g{h}\cap\g{n})-\dim(\g{h}\cap\g{g}_{2\alpha})\geq 2n-1$.

Therefore, we have obtained $2n-1\leq\dim(\g{h}\cap\g{n})_{\g{g}_\alpha}\leq n-1$, which is a contradiction.
As a consequence, $\g{h}\cap\g{g}_{2\alpha}\neq 0$.
\end{proof}

In view of Lemma~\ref{lemma:vector-g2alpha} we assume from now on that there is a nonzero vector $Z\in\g{h}\cap\g{g}_{2\alpha}$.
We may further suppose that $\lvert Z\rvert^2=2$ (that is, $\lvert Z\rvert_{AN}=1$).
Recall also from Lemma~\ref{lemma:B+galpha} that we may assume, up to orbit equivalence, that there is a vector $\xi_\alpha\in\g{g}_\alpha$ such that $B+(1-\theta)\xi_\alpha\in\g{h}_{\g{p}}^\perp$.

Under these assumptions we will show that

\begin{proposition}\label{prop:section-in-a+p1}
$\g{h}_{\g{p}}^\perp\subseteq\g{a}\oplus\g{p}_\alpha$.
\end{proposition}

In order to prove Proposition~\ref{prop:section-in-a+p1} we consider $(\g{h}_{\g{p}}^\perp)_{\g{p}_{2\alpha}}$, the orthogonal projection of $\g{h}_{\g{p}}^\perp$ onto $\g{p}_{2\alpha}$.
It suffices to show that $\dim(\g{h}_{\g{p}}^\perp)_{\g{p}_{2\alpha}}=0$.
We carry this out in two steps.  
The first of those is relatively easy to obtain:

\begin{lemma}
$\dim(\g{h}_{\g{p}}^\perp)_{\g{p}_{2\alpha}}\leq 1$.
\end{lemma}

\begin{proof}
It follows from Lemma~\ref{lemma:vector-g2alpha} that $\dim(\g{h}_{\g{p}}^\perp)_{\g{p}_{2\alpha}}\leq 2$.
Thus we have to prove that the case $\dim(\g{h}_{\g{p}}^\perp)_{\g{p}_{2\alpha}}= 2$ is not possible.

Assume $\dim(\g{h}_{\g{p}}^\perp)_{\g{p}_{2\alpha}}= 2$.
Then we can write
\[
\g{h}_{\g{p}}^\perp=\R(B+(1-\theta)\xi_\alpha)\oplus(1-\theta)\g{w}\oplus\R(1-\theta)(\eta_\alpha+\eta_{2\alpha})
\oplus\R(1-\theta)(\phi_\alpha+\phi_{2\alpha}),
\]
where $\g{w}\subseteq\g{g}_\alpha$, $\eta_\alpha$, $\phi_\alpha\in\g{g}_\alpha$, $\eta_{2\alpha}$, $\phi_{2\alpha}\in\g{g}_{2\alpha}$, and we can assume that $(\eta_{2\alpha},\phi_{2\alpha},Z)$ is a positively oriented orthonormal basis of $\g{g}_{2\alpha}$ (with respect to the inner product of $\g{a}\oplus\g{n}$).

Since $\g{h}_{\g{p}}^\perp$ is totally real,
\begin{align*}
0
&{}=\langle J_Z(1-\theta)(\eta_\alpha+\eta_{2\alpha}),(1-\theta)(\phi_\alpha+\phi_{2\alpha})\rangle\\
&{}=\langle (1-\theta)(J_Z \eta_\alpha-\phi_{2\alpha}),(1-\theta)(\phi_\alpha+\phi_{2\alpha})\rangle\\
&{}=2\langle J_Z\eta_\alpha,\phi_\alpha\rangle-2\langle\phi_{2\alpha},\phi_{2\alpha}\rangle
=2\langle [\eta_\alpha,\phi_\alpha],Z\rangle-4,
\end{align*}
and we get $\langle [\eta_\alpha,\phi_\alpha],Z\rangle=2$.

Since $\g{h}$ is orthogonal to $[\g{h}_{\g{p}}^\perp,\g{h}_{\g{p}}^\perp]$ by Proposition~\ref{th:criterion},
we also have
\begin{align*}
0
&{}=\langle Z,[(1-\theta)(\eta_\alpha+\eta_{2\alpha}),(1-\theta)(\phi_\alpha+\phi_{2\alpha})]\rangle\\
&{}=\bigl\langle Z,(1+\theta)\bigl([\eta_\alpha,\phi_\alpha]-[\theta\eta_\alpha,\phi_\alpha]
-[\theta\eta_\alpha,\phi_{2\alpha}]-[\theta\eta_{2\alpha},\phi_\alpha]
-[\theta\eta_{2\alpha},\phi_{2\alpha}]\bigr)\bigr\rangle\\
&{}=\langle Z,[\eta_\alpha,\phi_\alpha]\rangle=2,
\end{align*}
which gives a contradiction and finishes the proof.
\end{proof}

The second step in the proof of Proposition~\ref{prop:section-in-a+p1}
is to show that the case $\dim(\g{h}_{\g{p}}^\perp)_{\g{p}_{2\alpha}}= 1$ is not possible.
To do this we will need the following lemmas:

\begin{lemma}\label{lemma:spperp}
Assume $\dim(\g{h}_\g{p}^\perp)_{\g{p}_{2\alpha}}=1$.
We can write
\[
\g{h}_{\g{p}}^\perp
=\R(B+(1-\theta)\xi_\alpha)\oplus(1-\theta)\g{w}\oplus\R(1-\theta)(\eta_\alpha+\eta_{2\alpha}),
\]
where $\g{w}$ is a totally real subspace of $\g{g}_\alpha$, $\xi_\alpha$ and $\eta_\alpha$ are linearly independent vectors of $\g{g}_\alpha$, $\eta_{2\alpha}\in\g{g}_{2\alpha}$ with $\lvert\eta_{2\alpha}\rvert^2=2$, $\H\xi_\alpha+\H\eta_\alpha$ is orthogonal to $\H\g{w}$, 
$2[\xi_\alpha,\eta_\alpha]=-\eta_{2\alpha}$,
and
\[
\eta_\alpha
=\frac{\langle\xi_\alpha,\eta_\alpha\rangle}{\lvert\xi_\alpha\rvert^2}\xi_\alpha
-\frac{1}{\lvert\xi_\alpha\rvert^2}J_{\eta_{2\alpha}}\xi_\alpha.
\]
\end{lemma}

\begin{proof}
By Lemma~\ref{lemma:B+galpha} and the assumption $\dim(\g{h}_{\g{p}}^\perp)_{\g{p}_{2\alpha}}= 1$, it is clear that we can write
$\g{h}_{\g{p}}^\perp
=\R(B+(1-\theta)\xi_\alpha)\oplus(1-\theta)\g{w}\oplus\R(1-\theta)(\eta_\alpha+\eta_{2\alpha})$,
where $\g{w}\subseteq\g{g}_\alpha$, $\xi_\alpha$, $\eta_\alpha\in\g{g}_\alpha$ are orthogonal to $\g{w}$, $\eta_{2\alpha}\in\g{g}_{2\alpha}$, and $\lvert\eta_{2\alpha}\rvert^2=2$ (thus, $\lvert\eta_{2\alpha}\rvert_{AN}^2=1$).

Since $\g{h}_{\g{p}}^\perp$ is totally real, it is clear that so is $\g{w}$.
Moreover, for each $U\in\g{w}$ and $X\in\g{g}_{2\alpha}$ we have
\[
0
=\langle J_X(B+(1-\theta)\xi_\alpha),(1-\theta)U\rangle
=\Bigl\langle (1-\theta)\Bigl(\frac{1}{2}X+J_X\xi_\alpha\Bigr),(1-\theta)U\Bigr\rangle
=2\langle J_X\xi_\alpha,U\rangle,
\]
which implies $\H\xi_\alpha$ is orthogonal to $\g{w}$.
Similarly, since $J_X\eta_{2\alpha}\in\g{a}\oplus\g{g}_{2\alpha}$ we get
\[
0
=\langle J_X(1-\theta)(\eta_\alpha+\eta_{2\alpha}),(1-\theta)U\rangle
=2\langle J_X\eta_\alpha,U\rangle,
\]
and we also obtain that $\H\eta_\alpha$ is orthogonal to $\g{w}$.
Finally,
\begin{align*}
0
&{}=\langle J_{X}(B+(1-\theta)\xi_\alpha),(1-\theta)(\eta_\alpha+\eta_{2\alpha})\rangle\\
&{}=\Bigl\langle (1-\theta)\Bigl(\frac{1}{2}X+J_{X}\xi_\alpha\Bigr),
(1-\theta)(\eta_\alpha+\eta_{2\alpha})\Bigr\rangle\\
&{}=2\langle J_{X}\xi_\alpha,\eta_\alpha\rangle+\langle X,\eta_{2\alpha}\rangle
=\langle 2[\xi_\alpha,\eta_\alpha]+\eta_{2\alpha},X\rangle,
\end{align*}
so $2[\xi_\alpha,\eta_\alpha]=-\eta_{2\alpha}$.
In particular, $\xi_\alpha$ and $\eta_\alpha$ are linearly independent vectors.

Let $U\in\g{g}_\alpha\ominus\g{w}$.
The vector $-\langle U,\xi_\alpha\rangle B+U-\frac{1}{2}\langle U,\eta_\alpha\rangle\eta_{2\alpha}$
is orthogonal to $\g{h}_{\g{p}}^\perp$.
Thus, for each $U\in\g{g}_\alpha\ominus\g{w}$ there exists $T_U\in\g{t}$ such that
\begin{equation}\label{eq:tangent-vectors}
T_U-\langle U,\xi_\alpha\rangle B+U-\frac{1}{2}\langle U,\eta_\alpha\rangle\eta_{2\alpha}\in\g{h}.
\end{equation}
Taking $U=\xi_\alpha$ and $U=\eta_\alpha$ in the previous expression, and then bracketing the result gives
\begin{align*}
&
\Bigl[T_{\xi_\alpha}-\lvert\xi_\alpha\rvert^2 B+\xi_\alpha
-\frac{1}{2}\langle \xi_\alpha,\eta_\alpha\rangle\eta_{2\alpha},\,
T_{\eta_\alpha}-\langle \eta_\alpha,\xi_\alpha\rangle B+\eta_\alpha
-\frac{1}{2}\lvert\eta_\alpha\rvert^2\eta_{2\alpha}\Bigr]\\
&\qquad{}=\frac{1}{2}\langle\xi_\alpha,\eta_\alpha\rangle\xi_\alpha
-\frac{1}{2}\lvert\xi_\alpha\rvert^2\eta_\alpha
+[T_{\xi_\alpha},\eta_\alpha]-[T_{\eta_\alpha},\xi_\alpha]\\
&\qquad\phantom{{}=}{}+\frac{1}{2}\bigl(\lvert\xi_\alpha\wedge\eta_\alpha\rvert^2-1\bigr)\eta_{2\alpha}
-\frac{1}{2}\lvert\eta_\alpha\rvert^2[T_{\xi_\alpha},\eta_{2\alpha}]
+\frac{1}{2}\langle\xi_\alpha,\eta_\alpha\rangle[T_{\eta_\alpha},\eta_{2\alpha}],
\end{align*}
which is an element of $\g{h}$.
Hence, taking inner product with the vectors $B+(1-\theta)\xi_\alpha$ and $(1-\theta)(\eta_\alpha+\eta_{2\alpha})\in\g{h}_{\g{p}}^\perp$ we deduce that
\begin{align}\label{eq:Txieta}
\langle[T_{\xi_\alpha},\eta_\alpha],\xi_\alpha\rangle
&{}=0,&
\langle[T_{\eta_\alpha},\xi_\alpha],\eta_\alpha\rangle
&{}=\frac{1}{2}\lvert\xi_\alpha\wedge\eta_\alpha\rvert^2-1.
\end{align}

We also calculate
\begin{equation}\label{eq:bracket-spperp}
\begin{aligned}
&{}[B+(1-\theta)\xi_\alpha,(1-\theta)(\eta_\alpha+\eta_{2\alpha})]\\
&\qquad{}=(1+\theta)\Bigl(\frac{1}{2}\eta_\alpha+\eta_{2\alpha}
+[\xi_\alpha,\eta_\alpha]-[\theta\xi_\alpha,\eta_\alpha]
-[\theta\xi_\alpha,\eta_{2\alpha}]\Bigr)\\
&\qquad{}=(1+\theta)\Bigl(-[\theta\xi_\alpha,\eta_\alpha]
+\frac{1}{2}\eta_\alpha+J_{\eta_{2\alpha}}\xi_\alpha+\frac{1}{2}\eta_{2\alpha}\Bigr).
\end{aligned}
\end{equation}
Since $\g{h}$ is orthogonal to $[\g{h}_{\g{p}}^\perp,\g{h}_{\g{p}}^\perp]$ by Proposition~\ref{th:criterion}, taking inner product with the element $T_{\eta_\alpha}-\langle\xi_\alpha,\eta_\alpha\rangle B+\eta_\alpha-\frac{1}{2}\lvert\eta_\alpha\rvert^2\eta_{2\alpha}\in\g{h}$,
and using $\langle J_{\eta_{2\alpha}}\xi_\alpha,\eta_\alpha\rangle=\langle[\xi_\alpha,\eta_\alpha],\eta_{2\alpha}\rangle=-1$ and~\eqref{eq:Txieta}, gives
\begin{align*}
0&{}=-\langle T_{\eta_\alpha},(1+\theta)[\theta\xi_\alpha,\eta_\alpha]\rangle
+\frac{1}{2}\lvert\eta_\alpha\rvert^2
+\langle J_{\eta_{2\alpha}}\xi_\alpha,\eta_\alpha\rangle
-\frac{1}{2}\lvert\eta_{\alpha}\rvert^2\\
&{}=-2\langle[T_{\eta_{\alpha}},\xi_\alpha],\eta_\alpha\rangle-1
=-\lvert\xi_\alpha\wedge\eta_\alpha\rvert^2+1.
\end{align*}
Using these results and 
$\lvert J_{\eta_{2\alpha}}\xi_\alpha\rvert^2 =\frac{1}{2}\lvert\eta_{2\alpha}\rvert^2\lvert\xi_\alpha\rvert^2 =\lvert\xi_\alpha\rvert^2$ we calculate
\begin{align*}
&\lvert\eta_\alpha\wedge(\langle\xi_\alpha,\eta_\alpha\rangle\xi_\alpha-J_{\eta_{2\alpha}}\xi_\alpha)\lvert^2\\
&\quad{}=\lvert\eta_\alpha\rvert^2
\lvert\langle\xi_\alpha,\eta_\alpha\rangle\xi_\alpha-J_{\eta_{2\alpha}}\xi_\alpha\rvert^2
-\langle\eta_\alpha,\langle\xi_\alpha,\eta_\alpha\rangle\xi_\alpha-J_{\eta_{2\alpha}}\xi_\alpha\rangle^2\\
&\quad{}=\lvert\eta_\alpha\rvert^2
(\langle\xi_\alpha,\eta_\alpha\rangle^2\lvert\xi_\alpha\rvert^2+\lvert J_{\eta_{2\alpha}}\xi_\alpha\rvert^2)
-(\langle\xi_\alpha,\eta_\alpha\rangle^2-\langle J_{\eta_{2\alpha}}\xi_\alpha,\eta_\alpha\rangle)^2\\
&\quad{}=\lvert\eta_\alpha\rvert^2
(\langle\xi_\alpha,\eta_\alpha\rangle^2\lvert\xi_\alpha\rvert^2+\lvert \xi_\alpha\rvert^2)
-(\langle\xi_\alpha,\eta_\alpha\rangle^2+1)^2
=(\lvert\xi_\alpha\wedge\eta_\alpha\rvert^2-1)(\langle\xi_\alpha,\eta_\alpha\rangle^2+1)=0.
\end{align*}
This implies that $\eta_\alpha$ and $\langle\xi_\alpha,\eta_\alpha\rangle\xi_\alpha-J_{\eta_{2\alpha}}\xi_\alpha$ are collinear, so we may write
\begin{equation*}
\eta_\alpha=\mu(\langle\xi_\alpha,\eta_\alpha\rangle\xi_\alpha-J_{\eta_{2\alpha}}\xi_\alpha), \quad \mu\in \R.
\end{equation*}
Taking inner product with $\eta_\alpha$ and recalling that
$\lvert\xi_\alpha\wedge\eta_\alpha\rvert=1=-\langle J_{\eta_{2\alpha}}\xi_\alpha,\eta_\alpha\rangle$,
the formula for $\eta_\alpha$ readily follows.\qedhere
\end{proof}

Setting $U=J_{\eta_{2\alpha}}\xi_\alpha$ in~\eqref{eq:tangent-vectors} we have that
\[
T_{J_{\eta_{2\alpha}}\xi_\alpha}+J_{\eta_{2\alpha}}\xi_\alpha
-\frac{1}{2}\langle J_{\eta_{2\alpha}}\xi_\alpha,\eta_\alpha\rangle\eta_{2\alpha}
=S+J_{\eta_{2\alpha}}\xi_\alpha
+\frac{1}{2}\eta_{2\alpha}\in\g{h},
\]
where we have defined $S=T_{J_{\eta_{2\alpha}}\xi_\alpha}$ for simplicity, and used Lemma~\ref{lemma:spperp} for the last summand.

\begin{lemma}\label{lemma:bracket-xialpha-S}
Under the above hypotheses, we have
\[
[\xi_\alpha,S]=\frac{\lvert\xi_\alpha\rvert^2}{2}J_{\eta_{2\alpha}}\xi_\alpha.
\]
\end{lemma}

\begin{proof}
Since $\g{k}_0$ stabilizes $\g{g}_\alpha$ and $\ad(S)$ is skew-symmetric, we get $[\xi_\alpha,S]\in\g{g}_\alpha\ominus\R\xi_\alpha$.
	
Let $W\in\g{w}$.
Since $\g{h}$ is orthogonal to $[\g{h}_{\g{p}}^\perp,\g{h}_{\g{p}}^\perp]$ 
and $\H\xi_\alpha$ is orthogonal to $\g{w}$, 
we have
\begin{align*}
0&{}=\Bigl< S+J_{\eta_{2\alpha}}\xi_\alpha+\frac{1}{2}\eta_{2\alpha},\,
[B+(1-\theta)\xi_\alpha,(1-\theta)W]\Bigr>\\
&{}=\Bigl< S+J_{\eta_{2\alpha}}\xi_\alpha+\frac{1}{2}\eta_{2\alpha},\,
(1+\theta)\Bigl(\frac{1}{2}W-[\theta\xi_\alpha,W]\Bigr)\Bigr>\\[1ex]
&{}=-\langle S,(1+\theta)[\theta\xi_\alpha,W]\rangle
=-2\langle S,[\theta\xi_\alpha,W]\rangle=2\langle [\xi_\alpha,S],W\rangle,
\end{align*}
or, in other words, $\langle[\xi_\alpha,S],\g{w}\rangle=0$.

Let $U\in\g{g}_\alpha\ominus(\g{w}\oplus\R\xi_\alpha\oplus\R J_{\eta_{2\alpha}}\xi_\alpha)$.
Then, there exists $T_U\in\g{t}$ such that $T_U+U\in\g{h}$.

Since $\g{h}$ is perpendicular to $[\g{h}_{\g{p}}^\perp,\g{h}_{\g{p}}^\perp]$, taking inner product of $T_U+U$ with the expression in~\eqref{eq:bracket-spperp}, and using Lemma~\ref{lemma:spperp} yields
\[
0=-\langle T_U,(1+\theta)[\theta\xi_\alpha,\eta_\alpha]\rangle
=-2\langle[T_U,\xi_\alpha],\eta_\alpha\rangle
=\frac{2}{\lvert\xi_\alpha\rvert^2}\langle[T_U,\xi_\alpha],J_{\eta_{2\alpha}}\xi_\alpha\rangle,
\]
from where $\langle[T_U,\xi_\alpha],J_{\eta_{2\alpha}}\xi_\alpha\rangle=0$.

Since $\g{h}$ is a Lie algebra,
\begin{equation}\label{eq:bracket-U-xialpha}
[S,U]-[T_U,J_{\eta_{2\alpha}}\xi_\alpha]+[J_{\eta_{2\alpha}}\xi_\alpha,U]-\frac{1}{2}[T_U,\eta_{2\alpha}]
=\Bigl[S+J_{\eta_{2\alpha}}\xi_\alpha+\frac{1}{2}\eta_{2\alpha},T_U+U\Bigr]\in\g{h}.
\end{equation}
Taking inner product of~\eqref{eq:bracket-U-xialpha} with $B+(1-\theta)\xi_\alpha\in\g{h}_{\g{p}}^\perp$ yields
\[
0=\langle[S,U],\xi_\alpha\rangle-\langle[T_U,J_{\eta_{2\alpha}}\xi_\alpha],\xi_\alpha\rangle
=\langle[\xi_\alpha,S],U\rangle.
\]
Therefore, $[\xi_\alpha,S]\in\R J_{\eta_{2\alpha}}\xi_\alpha$.

We finally take inner product of $S+J_{\eta_{2\alpha}}\xi_\alpha+\frac{1}{2}\eta_{2\alpha}$ with~\eqref{eq:bracket-spperp} taking into account Lemma~\ref{lemma:spperp}, the skew-symmetry of $\ad(S)$, and $\langle J_{\eta_{2\alpha}}\xi_\alpha,\eta_\alpha\rangle=-1$, to obtain
\begin{align*}
0&{}=-\langle S,(1+\theta)[\theta\xi_\alpha,\eta_\alpha]\rangle
+\frac{1}{2}\langle J_{\eta_{2\alpha}}\xi_\alpha,\eta_\alpha\rangle
+\lvert J_{\eta_{2\alpha}}\xi_\alpha\rvert^2+\frac{1}{4}\lvert\eta_{2\alpha}\rvert^2\\
&{}=-2\langle[S,\xi_\alpha],\eta_\alpha\rangle+\lvert\xi_\alpha\rvert^2
=\frac{2}{\lvert\xi_\alpha\rvert^2}\langle[S,\xi_\alpha],J_{\eta_{2\alpha}}\xi_\alpha\rangle
+\lvert\xi_\alpha\rvert^2,
\end{align*}
and thus,
\[
\langle[\xi_\alpha,S],J_{\eta_{2\alpha}}\xi_\alpha\rangle=\frac{\lvert\xi_\alpha\rvert^4}{2}.
\]
Since, $\lvert J_{\eta_{2\alpha}}\xi_\alpha\rvert^2=\lvert\xi_\alpha\rvert^2$, altogether this implies the result.
\end{proof}

We are now ready to prove the main step of this section.

\begin{proof}[Proof of Proposition~\ref{prop:section-in-a+p1}]
We just have to show that the case $\dim(\g{h}_{\g{p}}^\perp)_{\g{p}_{2\alpha}}= 1$ is not possible.
On the contrary, assume that this is true, and we will reach a contradiction.

We define 
\[
g=\Exp\Bigl(-\frac{2}{\lvert\xi_\alpha\rvert^2}\xi_\alpha\Bigr).
\]

The isotropy group of $H$ at $o$ is $H_o=H\cap K$, 
and its Lie algebra is $\g{h}\cap\g{k}=\g{h}\cap\g{t}$.
On the other hand, $H\cdot g(o)=g(g^{-1}Hg)\cdot o$, 
so the orbit of $H$ through $g(o)$ is congruent to the orbit of $g^{-1}Hg$ through $o$.
Thus, the isotropy group of $H$ at $g(o)$ is conjugate to the isotropy group of $g^{-1}Hg$ at $o$, 
and its Lie algebra is conjugate to $\Ad(g)(\g{h})\cap\g{k}=\Ad(g)(\g{h})\cap\g{t}$.

Since all the orbits of the action of $H$ are principal, the slice representations
of $H$ are trivial~\cite[Theorem~2.1.3]{BerndtConsoleOlmos16}.
Hence, at $o$ we have $[\g{h}\cap\g{t},\g{h}_{\g{p}}^\perp]=0$.
This implies that
$(1-\theta)[T,\xi_\alpha]=[T,B+(1-\theta)\xi_\alpha]=0$
for all $T\in\g{h}\cap\g{t}$.
Consequently, $\ad(\xi_\alpha)(\g{h}\cap \g{t})=0$ and $\Ad(g)(\g{h}\cap\g{t})=\g{h}\cap\g{t}$.
In particular, $\g{h}\cap\g{t}\subseteq\Ad(g)(\g{h})\cap\g{t}$.
The contradiction will be achieved if we show that the equality does not hold.

Recall the element $S+J_{\eta_{2\alpha}}\xi_\alpha
+\frac{1}{2}\eta_{2\alpha}\in\g{h}$.
Note that $S\not \in \g{h}\cap \g{t}$ by Lemma~\ref{lemma:bracket-xialpha-S}.
Using again Lemma~\ref{lemma:bracket-xialpha-S} and 
$[\xi_\alpha,J_{\eta_{2\alpha}}\xi_\alpha]=\frac{1}{2}\lvert\xi_\alpha\rvert^2\eta_{2\alpha}$ yields
\begin{align*}
&\Ad(g)\Bigl(S+J_{\eta_{2\alpha}}\xi_\alpha+\frac{1}{2}\eta_{2\alpha}\Bigr)
=e^{-\frac{2}{\lvert\xi_\alpha\rvert^2}\ad(\xi_\alpha)}
\Bigl(S+J_{\eta_{2\alpha}}\xi_\alpha+\frac{1}{2}\eta_{2\alpha}\Bigr)\\
&\qquad{}=S+J_{\eta_{2\alpha}}\xi_\alpha+\frac{1}{2}\eta_{2\alpha}
-\frac{2}{\lvert\xi_\alpha\rvert^2}[\xi_\alpha,\,
S+J_{\eta_{2\alpha}}\xi_\alpha]
+\frac{2}{\lvert\xi_\alpha\rvert^4}[\xi_\alpha,\,[\xi_\alpha,S]]\\
&\qquad{}=S+J_{\eta_{2\alpha}}\xi_\alpha+\frac{1}{2}\eta_{2\alpha}
-J_{\eta_{2\alpha}}\xi_\alpha
-\eta_{2\alpha}
+\frac{1}{2}\eta_{2\alpha}=S.
\end{align*}
As a consequence, $S\in\Ad(g)(\g{h})\cap\g{t}$, 
so $\dim\Ad(g)(\g{h})\cap\g{t}>\dim(\g{h}\cap\g{t})$, contradicting the fact that all isotropy groups are conjugate.
\end{proof}

Now that Proposition~\ref{prop:section-in-a+p1} has been proved,
we are in a position to verify Theorem~\ref{th:main}~(\ref{th:main-classification}) in the quaternionic case.

\begin{proof}[Proof of Theorem~\textup{\ref{th:main}~(\ref{th:main-classification})} when $\F=\H$]
Let $H$ be a closed connected subgroup of $\s{Sp}(1,n)$ acting polarly on $\H\s{H}^n$ inducing a foliation.
Because of Lemma~\ref{lemma:B+galpha} and Proposition~\ref{prop:section-in-a+p1}, we may assume up to orbit equivalence
that $\g{h}\subseteq\g{t}\oplus\g{a}\oplus\g{n}$,
$\g{h}_{\g{p}}^\perp\subseteq\g{a}\oplus\g{p}_\alpha$, and that there exists
some $\xi_\alpha\in \g{g}_\alpha$ such that $B+(1-\theta)\xi_\alpha\in\g{h}_\g{p}^\perp$.
Consequently, we may decompose $\g{h}_\g{p}^\perp=\R(B+(1-\theta)\xi_\alpha)\oplus(1-\theta)\g{v}$,
where $\g{v}\subseteq\g{g}_\alpha$ is orthogonal to $\R\xi_\alpha$.
Recalling that $\g{h}_\g{p}^\perp$ is a totally real subspace due to Proposition~\ref{prop:section-constant-curvature},
we deduce that $\g{v}$ is a totally real (equivalently, abelian)
subspace of $\g{g}_\alpha$, and the subspaces $\H\xi_\alpha$ and $\H\g{v}$
are orthogonal.

On the one hand, assume $\xi_\alpha=0$.
We then have $\g{h}_\g{p}^\perp=\g{a}\oplus(1-\theta)\g{v}$, so that
$\g{h}_{\g{a}\oplus\g{n}}=\g{n}\ominus\g{v}=\g{s}_{\g{a},\g{v}}$.
Using Proposition~\ref{prop:orbit-equivalence}, we obtain that $H$ and $\s{S}_{\g{a},\g{v}}$
have the same orbits.

On the other hand, suppose $\xi_\alpha\neq 0$.
In this case, we define the element
\[
g=\Exp\Bigl(-\frac{2}{\lvert\xi_\alpha\rvert^2}\xi_\alpha\Bigr)\in N.
\]
Then $H$ and $gHg^{-1}$ induce orbit equivalent actions on $\H\s{H}^n$.
Moreover, the Lie algebra of $gHg^{-1}$ is $\Ad(g)\g{h}\subseteq \g{t}\oplus\g{a}\oplus\g{n}$.
Our aim is to show that $(\Ad(g)\g{h})_{\g{p}}^\perp=(1-\theta)\g{w}$, where $\g{w}=\R\xi_\alpha\oplus\g{v}$
is an abelian subspace of $\g{g}_\alpha$.

Note that $B+\xi_\alpha$ is orthogonal to $\g{h}$ because
$B+(1-\theta)\xi_\alpha\in\g{h}_\g{p}^\perp$ and $\theta\g{n}\perp\g{h}$.
Consequently, $\Ad(g)\g{h}$ is orthogonal to the vector
\begin{align*}
    \Ad(g^{-1})^*(B+\xi_\alpha)
    &=
    e^{-\frac{2}{\lvert\xi_\alpha\rvert^2}\ad(\theta\xi_\alpha)}(B+\xi_\alpha)
    \equiv
    B+\xi_\alpha-\frac{2}{\lvert\xi_\alpha\rvert^2}\lvert\xi_\alpha\rvert^2H_\alpha
    =
    \xi_\alpha
    \pmod{\theta\g{n}},
\end{align*}
so $\xi_\alpha$ is orthogonal to $\Ad(g)\g{h}$.

Let $\eta\in\g{v}$ be arbitrary.
Then $\eta$ is orthogonal to $\g{h}$ since $(1-\theta)\eta\in\g{h}_\g{p}^\perp$ and $\g{h}\perp\theta\g{n}$.
Moreover, because the action of $H$ is polar, the element
\begin{align*}
    (1+\theta)\eta-4[\theta\xi_\alpha,\eta]=2[B+(1-\theta)\xi_\alpha,(1-\theta)\eta]
\end{align*}
is orthogonal to $\g{h}$, and since $\eta$ and $\theta\g{n}$ are already orthogonal
to $\g{h}$ we see that $[\theta\xi_\alpha,\eta]$ is orthogonal
to $\g{h}$.
We deduce that $\Ad(g)\g{h}$ is orthogonal to the vectors
\begin{equation*}
    \Ad(g^{-1})^*\eta
    =
    e^{-\frac{2}{\lvert\xi_\alpha\rvert^2}\ad(\theta\xi_\alpha)}\eta
    \equiv
    \eta-\frac{2}{\lvert\xi_\alpha\rvert^2}[\theta\xi_\alpha,\eta]
    \pmod{\theta\g{n}}
\end{equation*}
and
\begin{align*}
    \Ad(g^{-1})^*[\theta\xi_\alpha,\eta]
    &\equiv
    [\theta\xi_\alpha,\eta]
    \pmod{\theta\g{n}}.
\end{align*}
This implies that the vectors $\eta$ and $[\theta\xi_\alpha,\eta]$ are both orthogonal to $\Ad(g)\g{h}$.

Summarizing, we have seen that $\Ad(g)\g{h}$ is orthogonal to $\g{w}=\R\xi_\alpha\oplus\g{v}$,
so for dimension reasons we have $(\Ad(g)\g{h})_\g{p}^\perp=(1-\theta)\g{w}$ and 
$(\Ad(g)(\g{h}))_{\g{a}\oplus\g{n}}=\g{a}\oplus(\g{n}\ominus \g{w})=\g{s}_{0,\g{w}}$.
Proposition~\ref{prop:orbit-equivalence} then guarantees that $gHg^{-1}$ has the same orbits as
$\s{S}_{0,\g{w}}$, so we deduce that the actions of $H$ and $\s{S}_{0,\g{w}}$ are orbit equivalent.
This concludes the proof.\qedhere
\end{proof}

\section{Polar foliations on the Cayley hyperbolic plane}\label{sec:cayley}

We now classify polar homogeneous foliations on the Cayley hyperbolic plane $M=\O\s{H}^2=\s{F}_4^{-20}/\s{Spin}(9)$.

We consider a closed connected subgroup $H\subseteq \s{F}_4^{-20}$ that acts polarly on $M$ inducing a homogeneous foliation.
The action of $H$ is assumed to be nontrivial and nontransitive.
We let $\Sigma\subseteq M$ be the section through $o$ of the action, and let $\g{h}_{\g{p}}^\perp$ be its tangent space.
By virtue of Proposition~\ref{th:Borel:tan}, we can assume up to orbit equivalence that the Lie algebra $\g{h}$ of $H$ is contained in a maximally noncompact subalgebra of the form $\g{t}\oplus\g{a}\oplus\g{n}$, with $\g{t}\subseteq \g{k}_0$ an abelian subspace.

According to lemmas~\ref{lemma:equivalence-B+galpha} and~\ref{lemma:B+galpha} we can consider, up to orbit equivalence, two possibilities.

\subsection{The case $\g{a}\subseteq \g{h}_{\g{p}}^\perp$}\label{sec:a-in-spperp}\hfill

Let us suppose that $B\in\g{h}_{\g{p}}^\perp$.
In particular, we see that $\g{h}_{\g{p}}^\perp$ is invariant under the Jacobi operator $R_{B}=-\ad(B)^2\in\End(\g{p})$.
As a consequence, $\g{h}_{\g{p}}^\perp$ can be decomposed into the eigenspaces of the restriction $R_{B}\vert_{\g{h}_{\g{p}}^\perp}$, which are precisely the intersections of $\g{h}_{\g{p}}^\perp$ with the eigenspaces of $R_{B}$.
Thus, from~\eqref{eq:jacobi} we have the splitting
\begin{equation*}
	\g{h}_{\g{p}}^\perp=\g{a}\oplus(1-\theta)\g{v}\oplus(1-\theta)\g{w},
\end{equation*}
where $\g{v}$ is a real subspace of $\g{g}_\alpha$ and $\g{w}$ is a subspace of $\g{g}_{2\alpha}$.

We first show that $\g{v}$ (and therefore $\g{g}_{\alpha}\ominus\g{v}$) is invariant under $J_X$ for every $X\in\g{w}$.
Let $U\in\g{v}$ and $X\in\g{w}$.
Because the action of $H$ is polar, we see that $(1+\theta)J_X U=[(1-\theta)U,(1-\theta)X]$ is orthogonal to $\g{h}$, and since $\theta\g{n}$ is perpendicular to $\g{h}$, we conclude that $J_X U$ is perpendicular to $\g{h}$.
This gives $J_X U\in \g{v}$, as desired.

Now, note that the maximum dimension of a totally geodesic submanifold of $M$ is equal to $8= \dim \g{g}_\alpha$, 
corresponding to $\R\s{H}^8(1)\subseteq \O\s{H}^2$ or $\H\s{H}^2\subseteq\O\s{H}^2$.
Since the dimension of $\Sigma$ is $1+\dim \g{v}+\dim\g{w}$, we see that $\g{v}\neq \g{g}_\alpha$, so we can choose a nonzero vector $V\in\g{g}_\alpha\ominus\g{v}$.
Let $X\in\g{w}$ be any vector.
The elements $V$ and $J_X V$ belong to $\g{g}_\alpha\ominus \g{v}\subseteq \g{h}_{\g{a}\oplus\g{n}}$, the orthogonal projection of $\g{h}$ onto $\g{a}\oplus\g{n}$. 
Thus, there exist $T$, $T'\in\g{t}$ satisfying $T+V$, $T'+J_X V\in\g{h}$.
In particular, we see that $[T,J_X V]+[V,T']+(1/2)\lvert V \rvert^2 X=[T+V,T'+J_X V]\in\g{h}$.
Taking inner product with $(1-\theta)X\in\g{h}_{\g{p}}^\perp$ we deduce that $(1/2)\lvert V \rvert^2 \lvert X \rvert ^2 = 0$.
As a consequence, $X=0$, which means that $\g{w}$ is trivial.

Finally, we claim that $\g{v}$ is at most one-dimensional.
Indeed, suppose that $\g{v}\neq 0$.
Then the restriction of $R_{B}$ to $\g{h}_{\g{p}}^\perp\ominus \g{a}=(1-\theta)\g{v}$ is a multiple of the identity map.
As $\Sigma$ is a symmetric space of rank one, we deduce from this that $\Sigma$ is a space of constant curvature.
Now, the sectional curvature of $\Sigma$ is $-{1}/{4}$, and Subsection~\ref{sec:fhn-tg-submanifolds} reveals that $\Sigma$ is congruent to $\R\s{H}^2(2)$.
In particular, $\g{v}$ is one-dimensional.
In conclusion, the normal space $\g{h}_{\g{p}}^\perp$ is either $\g{a}$ or $\g{a}\oplus(1-\theta)\ell$ for a line $\ell\subseteq \g{g}_\alpha$.

We can summarize our discussion in the following result:

\begin{proposition}\label{prop:case-a-conclusion}
If $\g{a}\subseteq \g{h}_{\g{p}}^\perp $, then $\g{h}_{\g{a}\oplus\g{n}}=\g{n}\ominus\g{v}$, 
where $\g{v}\subseteq \g{g}_\alpha$ and $\dim \g{v}\leq 1$.
\end{proposition}

\subsection{The case $\g{h}_{\g{p}}^\perp\cap\g{p}_\alpha\neq 0$}\hfill

Now we assume that there exists a nonzero vector $\xi_\alpha\in\g{g}_\alpha$ with $(1-\theta)\xi_\alpha\in \g{h}_{\g{p}}^\perp$.
Without loss of generality, we may choose $\xi_\alpha$ so that $\lvert \xi_\alpha \rvert = 1$.
Once again, we know that $\g{h}_{\g{p}}^\perp$ is invariant under the Jacobi operator $R_{(1-\theta)\xi_\alpha}$, so using~\eqref{eq:jacobi} we may write
\begin{equation*}
	\g{h}_{\g{p}}^\perp=(1-\theta)(\R\xi_\alpha\oplus\g{v})\oplus \g{w},
\end{equation*}
where $\g{v}\subseteq \g{g}_\alpha\ominus \R \xi_\alpha=\g{J}\xi_\alpha$ is a real subspace and $\g{w}\subseteq \g{a}\oplus\g{p}_{2\alpha}$.

We start by claiming that the space $\R\xi_\alpha\oplus\g{v}$ is invariant under $J_X$ for all $X\in\g{w}_{\g{g}_{2\alpha}}$.
To see this, let $\eta\in\R\xi_\alpha\oplus\g{v}$ and $t B+(1-\theta)X\in\g{w}$, where $t\in \R$ and $X\in\g{g}_{2\alpha}$.
Since the action of $H$ is polar, we see that $(1+\theta)(\tfrac{t}{2}\eta-J_X\eta)=[t B+(1-\theta)X,(1-\theta)\eta]$ is orthogonal to $\g{h}$.
This, combined with the fact that $\theta\g{n}$ and $\eta$ are already orthogonal to $\g{h}$, implies that $J_X\eta$ is orthogonal $\g{h}$, so $J_X \eta\in\R \xi_\alpha \oplus \g{v}$, as desired.
In particular, the skew-symmetry of $J_X$ also implies that $\g{g}_{\alpha}\ominus(\R\xi_\alpha\oplus\g{v})$ is $J_X$-invariant.
Because $\dim\Sigma\leq8=\dim\g{g}_\alpha$, either $\g{h}_{\g{p}}^\perp=\g{p}_\alpha$ or the orthogonal complement $\g{g}_\alpha\ominus(\R\xi_\alpha\oplus\g{v})$ is nonzero.

\begin{lemma}\label{lemma:spperp-not-p-alpha}
If $\g{h}_{\g{p}}^\perp\subseteq\g{p}_\alpha$, then $\dim \g{h}_{\g{p}}^\perp=1$.
In particular, $\g{h}_{\g{p}}^\perp$ is not equal to $\g{p}_\alpha$.
\end{lemma}

\begin{proof}
Assume $\dim \g{h}_{\g{p}}^\perp\geq 2$.
Since $\g{h}_{\g{p}}^{\perp}\subseteq \g{p}_{\alpha}$, we have $\g{a}\oplus\g{g}_{2\alpha}\subseteq\g{h}_{\g{a}\oplus\g{n}}$.
Applying Lemma~\ref{lemma:projectionImpliesContained} we deduce  $\g{g}_{2\alpha}\subseteq \g{h}$.
Now, because $\dim \g{h}_{\g{p}}^\perp\geq 2$ and 
$\g{g}_\alpha\ominus\R \xi_\alpha = \g{J}\xi_\alpha$, we may find a nonzero vector $X\in\g{g}_{2\alpha}$ such that $(1-\theta)J_X\xi_\alpha\in\g{h}_{\g{p}}^\perp$.
As a consequence, the vector 
$(1+\theta)(\tfrac{1}{2}X -[\theta \xi, J_X \xi])
=[(1-\theta)\xi_\alpha,(1-\theta)J_X \xi]$ is perpendicular to $\g{h}$ due to the polarity of the action.
This contradicts the fact that $\g{g}_{2\alpha}\subseteq \g{h}$, so necessarily $\dim \g{h}_{\g{p}}^\perp=1$.
\end{proof}

\begin{lemma}
The subspace $\g{w}$ is contained in $\g{a}$.
\end{lemma}

\begin{proof}
It suffices to prove that the projection of $\g{w}$ onto $\g{g}_{2\alpha}$ is trivial.
Because of Lemma~\ref{lemma:spperp-not-p-alpha} and its preceding discussion, we know that there exists a nonzero vector $V\in\g{g}_{\alpha}\ominus(\R\xi_\alpha\oplus \g{v})$.
Now, let $tH_\alpha+(1-\theta)X\in\g{w}$, where $t\in \R$ and $X\in\g{g}_{2\alpha}$.
Recall that $\g{g}_{\alpha}\ominus(\R\xi_\alpha\oplus\g{v})\subseteq \g{h}_{\g{a}\oplus\g{n}}$ is invariant under $J_X$, so we deduce that $V$, $J_X V\in \g{h}_{\g{a}\oplus\g{n}}$.
As a consequence, we may select vectors $T$, $T'\in\g{t}$ such that $T+V$, $T'+J_X V\in\g{h}$.
Therefore, we have $[T,J_X V]+[V,T']+\tfrac{1}{2}\lvert V \rvert^2 X=[T+V,T'+J_X V]\in\g{h}$, and we obtain
\begin{equation*}
0=\Bigl< [T,J_X V]+[V,T']+\frac{1}{2}\lvert V \rvert^2 X, tH_\alpha+(1-\theta)X \Bigr> = \frac{1}{2}\lvert V \rvert^2 \lvert X \rvert^2,
\end{equation*}
thus forcing $X=0$.
This proves the desired assertion.\qedhere
\end{proof}

Because $\g{w}$ is contained in the one-dimensional space $\g{a}$, we see that either $\g{w}=\g{a}$ or $\g{w}=0$.
The case $\g{w}=\g{a}$ was already dealt with in Subsection~\ref{sec:a-in-spperp}, 
so we may assume directly that $\g{w}=0$, giving $\g{h}_{\g{p}}^\perp=(1-\theta)(\R \xi_\alpha \oplus \g{v})\subseteq \g{p}_\alpha$.
Lemma~\ref{lemma:spperp-not-p-alpha} then implies that 
$\dim \g{h}_{\g{p}}^\perp=1$, so $\g{v}=0$ and 
$\g{h}_{\g{p}}^\perp=\R (1-\theta)\xi_\alpha$.
We therefore conclude with:

\begin{proposition}\label{prop:case-palpha-conclusion}
If $\g{h}_{\g{p}}^\perp\cap \g{p}_\alpha \neq 0$, 
then $\g{h}_{\g{a}\oplus\g{n}}=(\g{a}\ominus \g{b})\oplus (\g{n}\ominus \ell)$, 
where $\g{b}\subseteq \g{a}$ is a vector subspace and $\ell\subseteq \g{p}_\alpha$ is a line.
\end{proposition}

\subsection{End of proof}\hfill

We now verify Theorem~\ref{th:main}~(\ref{th:main-classification}) for
the Cayley hyperbolic plane.
This will conclude the proof of Theorem~\ref{th:main}.

\begin{proof}[Proof of Theorem~\textup{\ref{th:main}~(\ref{th:main-classification})} when $\F=\O$]
Suppose $H\subseteq \s{F}_4^{-20}$ is a closed connected subgroup that acts polarly on
$M=\O\s{H}^2$ inducing a foliation.
By combining Proposition~\ref{th:Borel:tan}, Lemmas~\ref{lemma:equivalence-B+galpha} and~\ref{lemma:B+galpha} and Propositions~\ref{prop:case-a-conclusion} and~\ref{prop:case-palpha-conclusion}, we deduce that (up to orbit equivalence) $\g{h}\subseteq\g{t}\oplus\g{a}\oplus\g{n}$
and the tangent space 
$\g{h}_{\g{a}\oplus\g{n}}$ is of the form $\g{s}_{\g{b},\g{v}}$ for a subspace $\g{b}\subseteq \g{a}$ and a subspace $\g{v}\subseteq \g{g}_\alpha$ of dimension at most one.
Therefore, $H$ falls under the hypotheses of Proposition~\ref{prop:orbit-equivalence} (with $\g{z}=\g{a}\ominus\g{b}$ and $\g{v}_\alpha=\g{v}$).
This guarantees that the orbits of $H$ are equal to the orbits of the connected subgroup 
$\s{S}_{\g{b},\g{v}}\subseteq G$ whose Lie algebra is $\g{s}_{\g{b},\g{v}}$.
\end{proof}


\begin{thebibliography}{99}
\bibitem{AlekseevskyDiScala} 
D. V. Alekseevsky, A. J. Di Scala: Minimal homogeneous submanifolds of symmetric spaces, Lie groups and symmetric spaces, 
\textit{Amer. Math. Soc. Transl. Ser. 2} \textbf{210} (2003), 11–25.

\bibitem{BerndtBruck01}
J.~Berndt, M.~Br\"uck:
Cohomogeneity one actions on hyperbolic spaces,
\textit{J. Reine Angew. Math.} {\bf 541} (2001), 209--235.

\bibitem{BerndtConsoleOlmos16}
J.~Berndt, S.~Console, C.~E.~Olmos: {\it Submanifolds and holonomy}, second edition, 
Monographs and Research Notes in Mathematics, CRC Press, Boca Raton, FL, 2016.

\bibitem{berndt-diaz-ramos}
J.~Berndt, J.~C.~D\'{\i}az-Ramos: Homogeneous polar foliations of complex hyperbolic spaces, 
\textit{Comm.\ Anal.\ Geom.} \textbf{20} (2012), 435--454.

\bibitem{BerndtDiazTamaru10}
J.~Berndt, J.~C.~D\'iaz-Ramos, H.~Tamaru: Hyperpolar homogeneous foliations on symmetric spaces of noncompact type, 
\textit{J.~Differential Geom.}\ \textbf{86} (2010), 191--235.

\bibitem{BerndtTamaru07}
J.~Berndt, H.~Tamaru:
Cohomogeneity one actions on noncompact symmetric spaces of rank one, 
\textit{Trans.\ Amer.\ Math.\ Soc.}\ \textbf{359} (2007), 3425--3438.

\bibitem{BerndtTamaru03}
J.~Berndt, H.~Tamaru: Homogeneous codimension one foliations on noncompact symmetric spaces,  
\textit{J.~Differential Geom.}\ \textbf{63} (2003), 1--40.

\bibitem{BerndtTricerriVanHecke} 
J. Berndt, F. Tricerri, L. Vanhecke: 
\textit{Generalized Heisenberg groups and Damek-Ricci harmonic spaces}, Springer-Verlag, Berlin, 1995. 

\bibitem{Biliotti06}
L.~Biliotti:
Coisotropic and polar actions on compact irreducible Hermitian symmetric spaces,
\textit{Trans. Amer. Math. Soc.} {\bf 358} (2006), no.~7, 3003--3022.

\bibitem{Bohm98}
C. B\"ohm: Inhomogeneous Einstein metrics on low-dimensional spheres and other low-dimensional spaces,
\textit{Invent. Math.} {\bf 134} (1998), no.~1, 145--176.

\bibitem{Borel50}
A. Borel: Le plan projectif des octaves et les sphères comme espaces homogènes, 
\textit{C.\ R.\ Acad.\ Sci.\ Paris} \textbf{230} (1950), 1378--1380.

\bibitem{Bryant97}
R.~L. Bryant: Metrics with exceptional holonomy,
\textit{Ann. of Math.} (2) {\bf 126} (1987), no.~3, 525--576.

\bibitem{cowling-dooley-koranyi-ricci}
M.~Cowling, A.~Dooley, A.~Kor\'{a}nyi, F.~Ricci: An approach to symmetric spaces of rank one via groups of Heisenberg type, 
\textit{J.\ Geom.\ Anal.} \textbf{8} (1998), 199--237.

\bibitem{cowling-h-type}
M.~Cowling, A.~H Dooley, A.~Kor\'{a}nyi, F.~Ricci: $H$-type groups and Iwasawa decompositions, 
\textit{Adv.\ Math.} \textbf{87} (1991), 1--41.

\bibitem{Dadok85}
J.~Dadok: Polar coordinates induced by actions of compact Lie groups, \textit{Trans.\ Amer.\ Math.\ Soc.} \textbf{288} (1985), 125--137.

\bibitem{DiazDominguezKollross20}
J.~C.~D\'{\i}az Ramos, M.~Dom\'{\i}nguez-V\'{a}zquez, A.~Kollross: Polar actions on complex hyperbolic spaces, 
\textit{Math.\ Z.} \textbf{287} (2017), 1183--1213.

\bibitem{DiazDominguezRodriguez21}
J.~C.~D\'{\i}az-Ramos, M.~Dom\'{\i}nguez-V\'{a}zquez, A.~Rodr\'{\i}guez-V\'{a}zquez: Homogeneous and inhomogeneous isoparametric hypersurfaces in rank one symmetric spaces, 
\textit{J.\ Reine Angew.\ Math.} \textbf{779} (2021), 189--222.

\bibitem{DiazLorenzo} 
J.~C.~D\'{\i}az-Ramos, J.~M.~Lorenzo-Naveiro: Codimension two polar homogeneous foliations on symmetric spaces of noncompact type, 
\textit{Adv.\ Math.} \textbf{472} (2025), Paper No.~110295, 31 pp.

\bibitem{Eberlein96}
P.~B. Eberlein: \emph{Geometry of nonpositively curved manifolds}, Chicago Lectures in Mathematics, University of Chicago Press, Chicago, IL, 1996.

\bibitem{Gorodski04}
C.~Gorodski: Polar actions on compact symmetric spaces which admit a totally geodesic principal orbit,
\textit{Geom.\ Dedicata} \textbf{103} (2004), 193--204.

\bibitem{GorodskiKollross16}
C.~Gorodski, A.~Kollross: Some remarks on polar actions,
\textit{Ann. Global Anal. Geom.} {\bf 49} (2016), no.~1, 43--58.

\bibitem{Gray} 
A. Gray: A note on manifolds whose holonomy group is a subgroup of $\mathsf{Sp}(n)\cdot\mathsf{Sp}(1)$, \textit{Michigan Math. J.} \textbf{16} (1969), 125--128.

\bibitem{GroveZiller}
K.~Grove, W.~Ziller:
Polar manifolds and actions,
\textit{J.\ Fixed Point Theory Appl.} \textbf{11} (2012), no. 2, 279--313.

\bibitem{Helgason01}
S.~Helgason: \emph{Differential geometry, Lie groups, and symmetric spaces}, 
Graduate Studies in Mathematics, 34, American Mathematical Society, Providence, RI, 2001.

\bibitem{knapp}
A.~W. Knapp: \emph{Lie groups beyond an introduction}, second edition, 
Progress in Mathematics, 140, Birkh\"auser Boston, Boston, MA, 2002.

\bibitem{Kollross02}
A. Kollross: A classification of hyperpolar and cohomogeneity one actions,
\textit{Trans. Amer. Math. Soc.} {\bf 354} (2002), no.~2, 571--612.

\bibitem{kollross-duality}
A.~Kollross: Duality of symmetric spaces and polar actions, 
\textit{J.\ Lie Theory} \textbf{21} (2011), 961--986.

\bibitem{kollross-octonions}
A.~Kollross: Octonions, triality, the exceptional Lie algebra F4 and polar actions on the Cayley hyperbolic plane, 
\textit{Internat.\ J.\ Math.} \textbf{31} (2020), 2050051, 28 pp.

\bibitem{KollrossLytchak13}
A. Kollross, A. Lytchak: Polar actions on symmetric spaces of higher rank,
\textit{Bull.\ Lond.\ Math.\ Soc.}\ \textbf{45} (2013), 341--350.

\bibitem{Mostow}
{G.~D.~Mostow}: On maximal subgroups of real Lie groups, 
\textit{Ann.\ of Math.\ (2)} \textbf{74} (1961), 503--517.

\bibitem{PalaisTerng87}
R.~S.~Palais, C.-L. Terng: A general theory of canonical forms,
{\it Trans. Amer. Math. Soc.} {\bf 300} (1987), no.~2, 771--789.

\bibitem{PodestaThorbergsson99}
F.~Podest\`{a}, G.~Thorbergsson: Polar actions on rank-one symmetric spaces, 
\textit{J.\ Differential Geom.} \textbf{53} (1999), 131--175.

\bibitem{SanmartinSolonenko25}
V.~Sanmart\'{i}n-L\'{o}pez, I.~Solonenko:
Classification of cohomogeneity-one actions on symmetric spaces of noncompact type,
\verb|arXiv:2501.05553 [math.DG]|.

\bibitem{Solonenko21}
I.~Solonenko:
Homogeneous codimension-one foliations on reducible symmetric spaces of noncompact type,
\verb|arXiv:2112.02189 [math.DG]|.

\bibitem{wolf}
J. A. Wolf: Elliptic spaces in Grassmann manifolds, 
\textit{Illinois J.\ Math.}\ \textbf{7} (1963),
447--462.

\bibitem{Wu}
B. Wu: Isoparametric submanifolds of hyperbolic spaces, 
\textit{Trans.\ Amer.\ Math.\ Soc.}\ \textbf{331} (1992), no.~2, 609--626.
\end{thebibliography}
\end{document}